\documentclass[
reqno]{amsart}
\usepackage{latexsym,amsmath,amssymb,amscd}
\usepackage[all]{xy}
\def\today{\ifcase \month \or
   January \or February \or March \or April \or
   May \or June \or July \or August \or
   September \or October \or November \or December \fi
   \space\number\day , \number\year}

\makeatletter
  \newcommand\@dotsep{4.5}
  \def\@tocline#1#2#3#4#5#6#7{\relax
     \ifnum #1>\c@tocdepth 
     \else
     \par \addpenalty\@secpenalty\addvspace{#2}%
     \begingroup \hyphenpenalty\@M
     \@ifempty{#4}{%
     \@tempdima\csname r@tocindent\number#1\endcsname\relax
        }{%
         \@tempdima#4\relax
           }%
      \parindent\z@ \leftskip#3\relax \advance\leftskip\@tempdima\relax
      \rightskip\@pnumwidth plus1em \parfillskip-\@pnumwidth
       #5\leavevmode\hskip-\@tempdima #6\relax
       \leaders\hbox{$\m@th
       \mkern \@dotsep mu\hbox{.}\mkern \@dotsep mu$}\hfill
       \hbox to\@pnumwidth{\@tocpagenum{#7}}\par
       \nobreak
        \endgroup
         \fi}
\makeatother

\begin{document}

\makeatletter
\@addtoreset{figure}{section}
\def\thefigure{\thesection.\@arabic\c@figure}
\def\fps@figure{h,t}
\@addtoreset{table}{bsection}

\def\thetable{\thesection.\@arabic\c@table}
\def\fps@table{h, t}
\@addtoreset{equation}{section}
\def\theequation{
\arabic{equation}}
\makeatother

\newcommand{\bfi}{\bfseries\itshape}

\newtheorem{theorem}{Theorem}
\newtheorem{acknowledgment}[theorem]{Acknowledgment}
\newtheorem{algorithm}[theorem]{Algorithm}
\newtheorem{axiom}[theorem]{Axiom}
\newtheorem{case}[theorem]{Case}
\newtheorem{claim}[theorem]{Claim}
\newtheorem{conclusion}[theorem]{Conclusion}
\newtheorem{condition}[theorem]{Condition}
\newtheorem{conjecture}[theorem]{Conjecture}
\newtheorem{construction}[theorem]{Construction}
\newtheorem{corollary}[theorem]{Corollary}
\newtheorem{criterion}[theorem]{Criterion}
\newtheorem{data}[theorem]{Data}
\newtheorem{definition}[theorem]{Definition}
\newtheorem{example}[theorem]{Example}
\newtheorem{lemma}[theorem]{Lemma}
\newtheorem{notation}[theorem]{Notation}
\newtheorem{problem}[theorem]{Problem}
\newtheorem{proposition}[theorem]{Proposition}
\newtheorem{question}[theorem]{Question}
\newtheorem{remark}[theorem]{Remark}
\newtheorem{setting}[theorem]{Setting}
\numberwithin{theorem}{section}
\numberwithin{equation}{section}

\newcommand{\todo}[1]{\vspace{5 mm}\par \noindent
\framebox{\begin{minipage}[c]{0.85 \textwidth}
\tt #1 \end{minipage}}\vspace{5 mm}\par}

\renewcommand{\1}{{\bf 1}}

\newcommand{\hotimes}{\widehat\otimes}

\newcommand{\Ad}{{\rm Ad}}
\newcommand{\Alt}{{\rm Alt}\,}
\newcommand{\Ci}{{\mathcal C}^\infty}
\newcommand{\comp}{\circ}
\newcommand{\wt}{\widetilde}

\newcommand{\C}{\text{\bf C}}
\newcommand{\D}{\text{\bf D}}
\newcommand{\Hb}{\text{\bf H}}
\newcommand{\N}{\text{\bf N}}
\newcommand{\R}{\text{\bf R}}
\newcommand{\T}{\text{\bf T}}
\newcommand{\Ub}{\text{\bf U}}

\newcommand{\ph}{\text{\bf P}}
\newcommand{\de}{{\rm d}}
\newcommand{\ev}{{\rm ev}}
\newcommand{\fimes}{\mathop{\times}\limits}
\newcommand{\id}{{\rm id}}
\newcommand{\ie}{{\rm i}}
\newcommand{\End}{{\rm End}\,}
\newcommand{\Gr}{{\rm Gr}}
\newcommand{\GL}{{\rm GL}}
\newcommand{\Hilb}{{\bf Hilb}\,}
\newcommand{\Hom}{{\rm Hom}}
\renewcommand{\Im}{{\rm Im}}
\newcommand{\Ker}{{\rm Ker}\,}
\newcommand{\Lie}{\textbf{L}}
\newcommand{\lf}{{\rm l}}
\newcommand{\Loc}{{\rm Loc}\,}
\newcommand{\pr}{{\rm pr}}
\newcommand{\Ran}{{\rm Ran}\,}
\newcommand{\supp}{{\rm supp}\,}

\newcommand{\Tr}{{\rm Tr}\,}
\newcommand{\Tran}{\textbf{Trans}}

\newcommand{\CC}{{\mathbb C}}
\newcommand{\RR}{{\mathbb R}}
\newcommand{\NN}{{\mathbb N}}

\newcommand{\G}{{\rm G}}
\newcommand{\U}{{\rm U}}
\newcommand{\Gl}{{\rm GL}}
\newcommand{\SL}{{\rm SL}}
\newcommand{\SU}{{\rm SU}}
\newcommand{\VB}{{\rm VB}}

\newcommand{\Ac}{{\mathcal A}}
\newcommand{\Bc}{{\mathcal B}}
\newcommand{\Cc}{{\mathcal C}}
\newcommand{\Dc}{{\mathcal D}}
\newcommand{\Ec}{{\mathcal E}}
\newcommand{\Fc}{{\mathcal F}}
\newcommand{\Gc}{{\mathcal G}}
\newcommand{\Hc}{{\mathcal H}}
\newcommand{\Kc}{{\mathcal K}}
\newcommand{\Nc}{{\mathcal N}}
\newcommand{\Oc}{{\mathcal O}}
\newcommand{\Pc}{{\mathcal P}}
\newcommand{\Qc}{{\mathcal Q}}
\newcommand{\Rc}{{\mathcal R}}
\newcommand{\Sc}{{\mathcal S}}
\newcommand{\Tc}{{\mathcal T}}
\newcommand{\Uc}{{\mathcal U}}
\newcommand{\Vc}{{\mathcal V}}
\newcommand{\Xc}{{\mathcal X}}
\newcommand{\Yc}{{\mathcal Y}}
\newcommand{\Zc}{{\mathcal Z}}
\newcommand{\Ag}{{\mathfrak A}}
\renewcommand{\gg}{{\mathfrak g}}
\newcommand{\hg}{{\mathfrak h}}
\newcommand{\mg}{{\mathfrak m}}
\newcommand{\nng}{{\mathfrak n}}
\newcommand{\pg}{{\mathfrak p}}
\newcommand{\Gg}{{\mathfrak g}}
\newcommand{\Lg}{{\mathfrak L}}
\newcommand{\Sg}{{\mathfrak S}}
\newcommand{\Ug}{{\mathfrak u}}

\markboth{}{}

\makeatletter
\title[Universal enveloping algebras in a topological setting]
{On universal enveloping algebras\\ in a topological setting
}
\author{Daniel Belti\c t\u a and Mihai Nicolae}
\address{Institute of Mathematics ``Simion
Stoilow'' of the Romanian Academy,
P.O. Box 1-764, Bucharest, Romania}
\email{Daniel.Beltita@imar.ro, beltita@gmail.com}
\email{mihaitaiulian85@yahoo.com}
\thanks{This research was partially supported from the Grant of the Romanian National Authority for Scientific Research, 
CNCS-UEFISCDI, project number PN-II-ID-PCE-2011-3-0131. 
The first-named author also acknowledges partial support
from the Project MTM2010-16679, DGI-FEDER, of the MCYT, Spain.}

\date{February 25, 2014}
\makeatother

\begin{abstract}
We establish the exponential law for suitably topologies on 
spaces of vector-valued smooth functions on topological groups, 
where smoothness is defined by using differentiability along continuous one-parameter subgroups. 
As an application, we investigate 
the canonical correspondences between the universal enveloping algebra, 
the invariant local operators, 
and the convolution algebra of distributions supported at the unit element
 of any finite-dimensional Lie group, 
when one passes from finite-dimensional Lie groups to pre-Lie groups.  
The latter class includes for instance 
any locally compact groups, Banach-Lie groups, additive groups underlying locally convex vector spaces, 
and also mapping groups consisting of rapidly decreasing Lie group-valued functions. 
\end{abstract}

\maketitle

\tableofcontents

\section{Introduction}

It is well-known that Lie theory and the related representation theory have been successfully 
extended much beyond 
the classical setting of finite-dimensional real Lie groups, 
and this research area now includes locally compact groups (\cite{HM07}, \cite{HM13}), 
Lie groups modeled on Banach spaces or even on locally convex spaces (\cite{KM97}, \cite{Be06}, \cite{Ne06}), 
and some other classes of topological groups which may not be locally compact 
(\cite{BCR81}, \cite{Gl02b}, \cite{HM05}). 
The differential calculus on topological groups, 
involving functions which are smooth along the one-parameter subgroups 
(Definition~\ref{aux4}), 
plays an important role for these extensions of Lie theory and 
has recently found remarkable applications also to supergroups and their representation theory (\cite{NS12}, \cite{NS13}).  
We have merely mentioned here a very few references that are closer related to the topics of our paper.
 
On the other hand, as one can see for instance in \cite{Wa72} or \cite{Pe94}, 
a key fact in harmonic analysis and representation theory is that 
the universal enveloping algebras of finite-dimensional Lie algebras can be realized 
by linear functionals or operators on spaces of smooth functions on the corresponding 
Lie groups, for instance as 
convolution algebras of distributions supported at the unit element 
or as invariant linear differential operators.  
It is then natural to seek for   
such realizations beyond the classical setting of finite dimensional Lie groups, 
with motivation coming from the representation theory of groups of the aforementioned types.  
In the present paper we begin an investigation on that question, 
oriented towards a pretty large class of topological groups which have sufficiently many 
one-parameter subgroups,  
namely the pre-Lie groups; 
see  Definition~\ref{Def3.1} and Examples \ref{ex1}--\ref{ex4} below.  
(A sequel paper will deal with the situation when the domain~$\RR$ of the one-parameter groups 
is replaced by suitable subsets of more general topological fields, 
to some extent in the spirit of \cite{BGN04}, \cite{BN05}, and \cite{Ber08}.) 

To this end,  
one needs a suitable notion of distributions with compact support, that is, continuous linear functionals on the space of smooth functions of the group under consideration. 
While spaces of smooth functions on any topological group 
were already studied in the earlier literature, 
one still needs to give these function spaces a topology adequate for the purposes 
of turning their topological duals into associative algebras 
which act on function spaces by the natural operation of convolution. 
It should be pointed out here that although the convolution of functions on a topological group 
requires some Haar measure on that group, 
this is not necessary for the convolution of functions with distributions or measures 
(see Definition~\ref{conv_def}).

One of the main technical novelties of our paper is 
the construction of a suitable topology on the space of smooth functions 
on any topological group and with values in any locally convex space $\Yc$, 
for which for arbitrary topological groups $G$ and $H$ the exponential law for smooth functions 
$$\Ci(H\times G,\Yc)\simeq\Ci(H,\Ci(G,\Yc)) $$
holds true (see Theorem~\ref{P 4.15} and Remark~\ref{law} below). 
By using that fact, we then prove that for any pre-Lie group $G$, 
the convolution with distributions with compact support 
(that is, linear functionals which are continuous for the aforementioned suitable topology) 
does preserve the space of smooth functions $\Ci(G)$ 
(Proposition~\ref{P 4.28}). 
By focusing on distributions supported at $\1\in G$, 
we can thus identify them with continuous linear operators on $\Ci(G)$ 
which commute with the left translations and are local, 
in the sense that they do not increase the support of functions 
(Theorem~\ref{P 4.29}). 
Recall that Peetre's theorem from \cite{Pe60} ensures that 
the local operators on smooth manifolds are precisely the differential operators, not necessarily of finite order.  
If $G$ is any finite-dimensional Lie group, 
then we recover the natural correspondence between the distributions supported at 
$\1\in G$ and the left invariant differential operators on~$G$. 

The topology that we introduce on any function space $\Ci(G,\Yc)$ agrees with 
the topology of uniform convergence of functions and their derivatives 
if $G$ is any finite-dimensional real Lie group.  
However,  
unlike the most constructions of similar topologies on spaces of test functions 
from the literature, 
our construction (Definition~\ref{distrib}) does not need the group $G$ to be locally compact. 
In fact, spaces of test functions, distributions, and universal enveloping algebras 
were already investigated 
on locally compact groups which were not necessarily Lie groups, for instance:  
\begin{itemize}
\item Basic distribution theory on abelian locally compact groups by using differentiability along one-parameter subgroups 
was developed in \cite{Ri53}. 
 \item Let $G$ be any topological group which is a projective limit of Lie groups. 
Under the additional hypotheses that $G$ is simply connected, locally compact, and separable, 
one endowed the space $\Ci(G)$ in \cite{Ka59}, \cite{Ka61}, \cite{Ma61}, \cite{Br61},  \cite[Sect. 2]{BC75} 
with the topology of a locally convex space,  
which is nuclear if and only if 
every quotient group of $G$ whose Lie algebra is finite-dimensional is necessarily a Lie group,  
as proved in \cite[S\"atze 3.3, 3.5]{BC75}. 
\item Some nuclear function spaces on locally compact groups 
that do not use approximations by Lie groups were constructed in \cite{Py74}.  
\item Universal enveloping algebras of separable locally compact groups which are projective limits of Lie groups 
were studied in \cite{Br61}, \cite{MM64}, and \cite{MM65}. 
\item More recently, differential operators and their relation to distributions and convolutions 
on locally compact groups were also studied in \cite{Ed88} and \cite{Ak95a}. 
\end{itemize}

Our article is organized as follows: 
In Section~\ref{Sect2} we provide some basic definitions and auxiliary results 
from the differential calculus on topological groups. 
Section~\ref{Sect3} introduces the convolution of smooth functions with compactly supported distributions 
and states one of the main problems which motivated the present investigation 
(Problem~\ref{probl1}).
Section~\ref{Sect4} is devoted to proving the exponential law for smooth functions on topological groups 
(Theorem~\ref{P 4.15}), which is our main technical result. 
Finally, in Section~\ref{Sect5} we use that technical result 
for establishing the structure of invariant local operators 
(Theorem~\ref{P 4.29}).

\subsection*{General notation}
Throughout the present paper we denote by $G$, $H$ arbitrary topological groups, unless otherwise mentioned. 
We will assume that the topology of any topological group is separated. 
For any topological spaces $T$ and $S$ we denote by $\Cc(T,S)$ the set of all continuous maps from $T$ into $S$. 

\section{Preliminaries}\label{Sect2}

This section presents some ideas and notions of Lie theory that play a key role in the present paper. 
Our basic references for Lie theory of topological groups are \cite{BCR81}, \cite{HM05}, and \cite{HM07}.

\subsection*{The adjoint action of a topological group}

Let $G$ be any topological group with 
the set of neighborhoods of $\1\in G$ denoted by $\Vc_G(\1)$. 
Define 
$$\Lg(G)=\{\gamma\in\Cc({\mathbb R},G)\mid 
(\forall t,s\in{\mathbb R})\quad \gamma(t+s)=\gamma(t)\gamma(s)\}.$$ 
We endow $\Lg(G)$ with the topology of uniform convergence on the compact subsets of~${\mathbb R}$. 
It can be described by neighborhood bases as follows. 
For arbitrary $n\in{\mathbb N}$ and $U\in\Vc_G(\1)$ denote 
$$W_{n,U}=\{(\gamma_1,\gamma_2)\in\Lg(G)\times\Lg(G)\mid 
(\forall t\in[-n,n])\quad \gamma_2(t)\gamma_1(t)^{-1}\in U\}. 
$$
For every $\gamma_1\in\Lg(G)$ 
define 
$W_{n,U}(\gamma_1)=\{\gamma_2\in\Lg(G)\mid (\gamma_1,\gamma_2)\in W_{n,U}\}$. 
Then there exists a unique topology on $\Lg(G)$ with the property 
that for each $\gamma\in\Lg(G)$ the family 
$\{W_{n,U}(\gamma)\mid n\in{\mathbb N}, \ U\in\Vc_G(\1)\}$ 
is a fundamental system of neighborhoods of~$\gamma$.

\begin{definition}
\normalfont
The \emph{adjoint action} of the topological group $G$ is the mapping 
$$\Ad_G\colon G\times\Lg(G)\to\Lg(G),\quad (g,\gamma)\mapsto\Ad_G(g)\gamma:=g\gamma(\cdot)g^{-1}. $$
Since the action of $G$ on itself by inner automorphisms 
$G\times G\to G$, $(g,h)\mapsto ghg^{-1}$, is continuous, 
it follows that the above mapping $\Ad_G$ indeed takes values in $\Lg(G)$ and is a group action. 
\end{definition}

We now recall the following result for later use: 

\begin{lemma}\label{0141}
The adjoint action of every topological group is a continuous mapping. 
\end{lemma}

\begin{proof}
See \cite[Lemma 0.1.4.1]{BCR81}. 
\end{proof}

\subsection*{Differentiability along one-parameter subgroups}

\begin{definition}\label{aux4}
\normalfont
Let $G$ be any topological group with an arbitrary open subset $V\subseteq G$ 
and $\Yc$ be any real locally convex space.  
If $\varphi\colon V\to\Yc$, $\gamma\in\Lg(G)$, and $g\in V$, then we denote 
\begin{equation}\label{aux4_eq1}
(D^\lambda_\gamma\varphi)(g)=\lim_{t\to 0}\frac{\varphi(g\gamma(t))-\varphi(g)}{t} 
\end{equation}
if the limit in the right-hand side exists. 

We define $\Cc^1(V,\Yc)$ as the set of all $\varphi\in\Cc(V,\Yc)$ 
for which the function 
$$
D^\lambda\varphi\colon V\times\Lg(G)\to\Yc,\quad 
(D^\lambda\varphi)(g;\gamma):=(D^\lambda_\gamma\varphi)(g) $$
is well defined and continuous. 
We also denote $D^\lambda\varphi=(D^\lambda)^{1}\varphi$. 

Now let $n\ge 2$ and assume the space $\Cc^{n-1}(V,\Yc)$ and the mapping 
$(D^\lambda)^{n-1}$ have been defined. 
Then we define $\Cc^n(V,\Yc)$ as the set of all functions $\varphi\in\Cc^{n-1}(V,\Yc)$ 
for which the function 
$$\begin{aligned}
(D^\lambda)^{n}\varphi\colon 
V\times \Lg(G)\times\cdots\times\Lg(G) & \to\Yc,\\
(g;\gamma_1,\dots,\gamma_n) & \mapsto (D^\lambda_{\gamma_n}(D^\lambda_{\gamma_{n-1}}\cdots(D^\lambda_{\gamma_1}\varphi)\cdots))(g)
\end{aligned} $$
is well defined and continuous. 

Moreover we define $\Ci(V,\Yc):=\bigcap\limits_{n\ge1}\Cc^n(V,\Yc)$. 
If $\Yc=\CC$, then we write simply $\Cc^n(G):=\Cc^n(V,\CC)$  
etc.,  
for $n=1,2,\dots,\infty$. 
\end{definition}

\begin{notation}
\normalfont
It will be convenient to use the notation 
$$D^\lambda_\gamma\varphi:= D^\lambda_{\gamma_n}(D^\lambda_{\gamma_{n-1}}\cdots(D^\lambda_{\gamma_1}\varphi)\cdots)\colon G\to\Yc$$
whenever $\gamma:=(\gamma_1,\dots,\gamma_n)\in \Lg(G)\times\cdots\times\Lg(G)$ and $\varphi\in\Cc^n(G,\Yc)$. 
\end{notation}

\subsection*{Some auxiliary facts}

For later use we record the following well-known facts. 

\begin{lemma}\label{L 4.2}
Let $X$ and $T$ be any topological spaces, $\Yc$ be any locally convex space, and 
$f:X\times T\to \Yc$ be any continuous function. 
Pick any point $x_0\in X$ and compact set $K\subseteq T$. 
Then for any continuous seminorm $\vert\cdot\vert$ on $\Yc$ we have
\begin{equation}\label{L 4.2_eq1}
\lim\limits_{x\to x_0} \sup_{t\in K} \lvert f(x,t)-f(x_0,t) \rvert=0.
\end{equation}
\end{lemma}

\begin{proof}
This result is well known and is related to the exponential law for continuous functions 
$\Cc(X\times T,\Yc)\simeq\Cc(X,\Cc(T,\Yc))$; 
see for instance \cite[Th. 4.21]{AD51}. 
\end{proof}

In the following lemma we record the continuity with respect to parameters 
for the weak integrals in locally convex spaces which may not be complete; 
see \cite{Gl02a} for a thorough discussion of that integral, 
related differential calculus, and their applications to Lie theory. 

\begin{lemma}\label{T 4.25}
Let $X$ be any topological space, $\Yc$ be any locally convex space, $a,b\in \RR$, $a<b$, and 
$f\colon X\times [a,b]\to\Yc$ be any continuous function with the property that for every $x\in X$ 
there exists the weak integral  
$h(x)=\int\limits_a^b f(x,t)\de t$. 
Then the function $h\colon X\to\Yc$ obtained in this way is continuous.
\end{lemma}

\begin{proof}
To prove that the function~$h$ is continuous, let $\vert\cdot\vert$ be any continuous seminorm on~$\Yc$. 
It follows by \cite[Lemma 1.7]{Gl02a} that we have 
$$(\forall x,y\in X)\quad \vert h(x)-h(y)\vert\le (b-a)\sup\limits_{t\in[a,b]}\vert f(x,t)-f(y,t)\vert $$
and now by using Lemma~\ref{L 4.2} we readily see that the function $h\colon X\to\Yc$ is continuous. 
\end{proof}

\begin{lemma}\label{L 4.26}
Let $H$ be any topological group and $h\in \Cc(H,\Yc)$.
If $X\in \Lg(H)$ and the derivative 
$D_X^\lambda h\colon H\to\Yc$ exists and is continuous,  
then there exists a continuous function 
$\chi\colon \RR\times H\to\Yc$ 
satisfying for arbitrary $g\in H$ the conditions  
$$(\forall t\in\RR)\quad h(gX(t))=h(g)+t(D_X^\lambda h)(g)+t\chi(t,g)$$
 and $\chi(0,g)=0$.
\end{lemma}

\begin{proof}
This follows by \cite[Lemma 2.5]{NS12}; see also \cite[Prop. 2.3]{BB11}. 
\end{proof} 

\section{Distributions with compact support, convolutions, and local operators}\label{Sect3}

In this section we give a precise statement 
of the problem that motivated the present paper; see Problem~\ref{probl1} below.  

\subsection*{Topologies on spaces of smooth functions} 
Spaces of smooth functions and their topologies play an important role 
in the theory of infinite-dimensional Lie groups modeled on locally convex spaces; 
see for instance \cite[Def. I.5.1]{Ne06}. 
We will now introduce a suitable topology on spaces of smooth functions on any topological group $G$, 
by using compact subsets of the space of one-parameter subgroups $\Lg(G)$ and its Cartesian powers. 
This topology turns out to be adequate for establishing the exponential law (Theorem~\ref{P 4.15} and Remark~\ref{law}). 

\begin{definition}\label{distrib}
\normalfont
Let $G$ be any topological group and denote 
$$(\forall k\ge 1)\quad \Lg^k(G):=\underbrace{\Lg(G)\times\cdots\times\Lg(G)}_{k\text{ times}}.$$
Pick any open set $V\subseteq G$. 
If $\Yc$ is any locally convex space, 
then for every $k\ge 1$, any compact subsets $K_1\subseteq \Lg^k(G)$ and $K_2\subseteq V$, 
and any continuous seminorm $\vert\cdot\vert$ on~$\Yc$ we 
define 
$$p^{\vert\cdot\vert}_{K_1,K_2}\colon\Ci(V,\Yc)\to[0,\infty),\quad 
p^{\vert\cdot\vert}_{K_1,K_2}(f)=\sup\{\vert (D^\lambda_\gamma f)(x)\vert\mid \gamma\in K_1,x\in K_2\}.$$
\emph{For the sake of simplicity we will always omit the seminorm $\vert\cdot\vert$ on~$\Yc$ 
from the above notation, by writing simply $p_{K_1,K_2}$ instead of $p^{\vert\cdot\vert}_{K_1,K_2}$.}

We endow the function space $\Ci(V,\Yc)$ with the locally convex topology defined by the family of 
these seminorms $p_{K_1,K_2}$ 
and the locally convex space obtained in this way will be denoted by~$\Ec(V,\Yc)$. 
If $\Yc=\CC$ then we write simply $\Ec(V):=\Ec(V,\Yc)$. 

We also denote by $\Ec'(G)$ the topological dual of $\Ec(G)$ endowed with the  
weak dual topology. 
This means that we have 
$$\Ec'(G)=\{u\colon\Ec(G)\to\CC\mid u\text{ is linear and continuous}\}$$
as a linear space, and this space of linear functionals is endowed with 
the locally convex topology defined by the family of seminorms 
$\{q_B\mid B\text{ 
finite }\subseteq\Ec(G)\}$, 
where for every 
finite subset $B\subseteq\Ec(G)$ we define the seminorm 
$$q_B\colon\Ec'(G)\to\CC,\quad q_B(u):=
\max_{f\in B}\vert u(f)\vert.$$ 
The elements of $\Ec'(G)$ will be called \emph{distributions with compact support} on $G$. 
\end{definition}

Before  to go further, we state an interesting problem 
related to the above definition. 

\begin{problem}\label{probl2}
Find conditions on the topological group $G$ ensuring that every closed bounded subset of 
the locally convex space $\Ec(G)$ is compact. 
\end{problem}

The above problem will not be addressed in the present paper.  
Let us just mention that it has an affirmative answer 
if $G$ is any finite-dimensional Lie group; see \cite{Eh56}. 

\begin{definition}
\normalfont
Assume the setting of Definition~\ref{distrib}. 
The \emph{support} of any $u\in\Ec'(G)$ is denoted by $\supp u$ and is defined as the set of all 
points $x\in G$ with the property that for every neighborhood $U$ of $x$ 
there exists $f\in\Ec(G)$ such that $\supp f\subseteq U$ and $u(f)\ne0$. 
\end{definition}

\begin{remark}
 \normalfont 
For every $u\in\Ec'(G)$, by using its continuity with respect to 
the topology of $\Ec(G)$ introduced in Definition~\ref{distrib}, 
it follows that there exist a positive constant $C>0$, an integer $k\ge 1$, 
and some compact subsets $K_1\subseteq \Lg^k(G)$ and $K_2\subseteq G$ 
for which 
$$(\forall f\in\Ec(G))\quad \vert u(f)\vert\le C p_{K_1,K_2}(f). $$
This implies $\supp u\subseteq K_2$, 
hence the set $\supp u$ is compact in~$G$, 
and this motivates the terminology introduced in Definition~\ref{distrib}. 
For every compact subset $K\subseteq G$ we denote 
$$\Ec'_K(G):=\{u\in\Ec'(G)\mid \supp u\subseteq K\}.$$
In the case $K=\{\1\}$ we will denote simply $\Ec'_{\1}(G):=\Ec'_{\{\1\}}(G)$. 
\end{remark}

\subsection*{Convolutions} 
We next wish to introduce the convolution of a smooth function with a distribution with compact support. 

\begin{definition}\label{conv_def}
\normalfont 
Let $G$ be any topological group.
For all $\varphi\in\Ec(G)$ define $\check{\varphi}\in\Ec(G)$ by 
$$(\forall x\in G)\quad \check{\varphi}(x):=\varphi(x^{-1}).$$
Then for every $u\in\Ec'(G)$ we define $\check{u}\in\Ec'(G)$ by 
$$(\forall \varphi\in\Ec(G))\quad \check{u}(\varphi):=u(\check{\varphi}).$$
Finally, for all  $\varphi\in\Ec(G)$ and $u\in\Ec'(G)$ we define their \emph{convolution} 
as the function  
$$\varphi\ast u\colon G\to\CC,\quad (\varphi\ast u)(x):=\check u(\varphi\circ L_x) $$
where for all $x\in G$ we define $L_x\colon G\to G$, $L_x(y):=xy$. 
\end{definition}

The above definition is clearly correct, 
in the sense that $\check{\varphi},\varphi\circ L_x\in\Ec(G)$ for all $x\in G$ and $\varphi\in\Ec(G)$, 
if $G$ is a Lie group (see also \cite{Eh56}).  
We will show in Propositions \ref{check_lemma} and \ref{transl_lemma} below 
that the definition is actually correct
for arbitrary topological groups. 
To this end we begin by the following simple computation. 

\begin{remark}\label{check_deriv}
\normalfont
If $\varphi\in\Cc^1(G,\Yc)$, $x\in G$, and $\gamma\in\Lg(G)$, then 
\allowdisplaybreaks
\begin{align}
(D^\lambda_\gamma\check{\varphi})(x)
&=\lim_{t\to 0}\frac{\varphi(\gamma(-t)\cdot x^{-1})-\varphi(x^{-1})}{t} \nonumber\\
&=-\lim_{t\to 0}\frac{\varphi(\gamma(t)\cdot x^{-1})-\varphi(x^{-1})}{t} \nonumber\\
&=-\lim_{t\to 0}\frac{\varphi(x^{-1}\cdot (\Ad_G(x)\gamma)(t))-\varphi(x^{-1})}{t} \nonumber\\
&=-(D^\lambda_{\Ad_G(x)\gamma}\varphi)(x^{-1}). \nonumber
\end{align}
Similarly, if $n\ge 1$, $\varphi\in\Ci(G,\Yc)$, and $\gamma_1,\dots,\gamma_n\in\Lg(G)$, 
then for all $x\in G$ we have 
$$(D^\lambda_{\gamma_1}\cdots D^\lambda_{\gamma_n})(x)
=(-1)^n(D^\lambda_{\Ad_G(x)\gamma_n}\cdots D^\lambda_{\Ad_G(x)\gamma_n}\varphi)(x^{-1})$$
(see \cite[pag. 45]{BCR81}). 
\end{remark}

\begin{proposition}\label{check_lemma}
If $G$ is any topological group, then for all $\varphi\in\Ec(G)$ we have $\check{\varphi}\in\Ec(G)$. 
Moreover, the mapping 
$\Ec(G)\to\Ec(G)$, $\varphi\mapsto\check{\varphi}$, 
is an isomorphism of locally convex spaces. 
\end{proposition}

\begin{proof} 
The linear map $\varphi\mapsto\check{\varphi}$ is equal to its own inverse, hence it suffices to prove that it is continuous. 
To this end define for arbitrary $n\ge 1$, 
$$\begin{aligned}
\Psi_n\colon &\Lg(G)\times\cdots\times\Lg(G)\times G\to\Lg(G)\times\cdots\times\Lg(G)\times G,\\
&\Psi(\gamma_1,\dots,\gamma_n,x)=(\Ad_G(x)\gamma_1,\dots,\Ad_G(x)\gamma_n,x^{-1}). 
\end{aligned}$$
It follows directly by Lemma~\ref{0141} that the mapping $\Psi$ is a homeomorphism. 
Moreover, by Remark~\ref{check_deriv}, it follows that for every $\varphi\in\Cc^n(G)$ we have 
\begin{equation}\label{check_lemma_proof_eq1}
(D^\lambda)^n\check{\varphi}=(-1)^n((D^\lambda)^n\varphi)\circ\Psi_n
\end{equation} 
hence $\check{\varphi}\in\Cc^n(G)$. 
Since $n\ge 1$ is arbitrary, this shows that if $\varphi\in\Ec(G)$, then $\check{\varphi}\in\Ec(G)$. 

To check that the linear mapping $\Ec(G)\to\Ec(G)$, $\varphi\mapsto\check{\varphi}$, is also continuous, 
let $k\ge 1$ be any integer and the compact sets $K_1\subseteq \Lg^k(G)$ and $K_2\subseteq G$ be arbitrary. 
Define 
$$K'_1:=\{(\Ad_G(x)\gamma_1,\dots,\Ad_G(x)\gamma_k)\mid x\in K_1,(\gamma_1,\dots,\gamma_k)\in K_1\}$$ 
and $K'_2:=\{x^{-1}\mid x\in K_2\}$. 
Since both the inversion mapping and the adjoint action of~$G$ are continuous (Lemma~\ref{0141}), 
it is easily seen that the sets $K'_1$ and $K'_2$ are compact. 
Moreover, it follows by \eqref{check_lemma_proof_eq1} along with Definition~\ref{distrib} that we have 
$$(\forall\varphi\in\Ec(G))\quad p_{K_1,K_2}(\check{\varphi})\le p_{K'_1,K'_2}(\varphi)$$
hence the linear mapping $\Ec(G)\to\Ec(G)$, $\varphi\mapsto\check{\varphi}$, is indeed continuous. 
\end{proof}

\begin{remark}\label{transl_deriv}
\normalfont
If $\varphi\in\Cc^1(G,\Yc)$, $x,g\in G$, and $\gamma\in\Lg(G)$, then 
$$(D^\lambda_\gamma(\varphi\circ L_x))(g)
=\lim_{t\to 0}\frac{\varphi(xg\gamma(t))-\varphi(xg)}{t}
=(D^\lambda_\gamma\varphi)(xg).$$
Therefore $D^\lambda_\gamma(\varphi\circ L_x)=(D^\lambda_\gamma\varphi)\circ L_x$.
\end{remark}

\begin{proposition}\label{transl_lemma}
If $G$ is any topological group and  $\Yc$ is any locally convex space, 
then for all $\varphi\in\Ci(G,\Yc)$ and $x\in G$ we have $\varphi\circ L_x\in\Ci(G,\Yc)$. 
\end{proposition}

\begin{proof}
For arbitrary $n\ge 1$ we have the homeomorphism
$$F_n^x\colon \Lg(G)\times\cdots\times\Lg(G)\times G\to\Lg(G)\times\cdots\times\Lg(G)\times G,\quad 
F_n^x(\gamma_1,\dots,\gamma_n,g)=(\gamma_1,\dots,\gamma_n,xg). $$
On the other hand, by iterating Remark~\ref{transl_deriv}, it follows that 
for every $\varphi\in\Ci(G,\Yc)$ we have $(D^\lambda)^n(\varphi\circ L_x)=((D^\lambda)^n\varphi)\circ F_n^x$ 
hence $(D^\lambda)^n(\varphi\circ L_x)$ is a continuous function. 
Since $n\ge 1$ is arbitrary, we obtain $\varphi\circ L_x\in\Ci(G,\Yc)$, and this completes the proof. 
\end{proof}

As already mentioned, the above Propositions \ref{check_lemma} and \ref{transl_lemma} 
imply in particular that Definition~\ref{conv_def} is correct. 
For later use we now record the version of these results for the multiplication map; 
see also Remark~\ref{obs 3.11} below. 

\begin{proposition}\label{multi}
If $G$ is any topological group with the multiplication  
$m\colon G\times G\to G$, $(x,y)\mapsto xy$, 
then for any locally convex space $\Yc$ the linear mapping 
$$\Ec(G,\Yc)\to\Ec(G\times G,\Yc),\quad \varphi\mapsto\varphi\circ m$$
is well-defined and continuous. 
\end{proposition}

\begin{proof}
Recall from \cite[pag. 46]{BCR81} that 
for every $\varphi\in\Ci(G,\Yc)$, $x,y\in G$, $k\ge 1$, and $\alpha_1,\dots,\alpha_k,\beta_1,\dots,\beta_k\in\Lg(G)$ 
we have 
$$\begin{aligned}
((D^\lambda)^k(\varphi &\circ m))(x,y;(\alpha_1,\beta_1),\dots,(\alpha_k,\beta_k)) \\
=&\sum_{\ell=0}^k\sum_{\stackrel{\scriptstyle i_1<\cdots<i_\ell}{i_{\ell+1}<\cdots<i_k}} 
((D^\lambda)^\ell\varphi)(xy;\beta_{i_1},\dots,\beta_{i_\ell},\Ad_G(y^{-1})\alpha_{i_{\ell+1}},\dots,\Ad_G(y^{-1})\alpha_{i_k})
\end{aligned} $$
where we assume $\{i_1,\dots,i_\ell,i_{\ell+1},\dots,i_k\}=\{1,\dots,k\}$. 
With this formula at hand, the continuity of the map $\varphi\mapsto\varphi\circ m$ can be checked 
just as the continuity of $\varphi\mapsto\check{\varphi}$ in the proof of Proposition~\ref{check_lemma}. 
\end{proof}

\subsection*{Algebras of local operators} 

\begin{definition}
\normalfont
Let $G$ be any topological group. 
A \emph{local operator} on $G$ is any continuous linear operator $D\colon \Ec(G)\to\Ec(G)$ 
with the property 
$$(\forall f\in\Ec(G))\quad \supp(Df)\subseteq\supp f. $$
We will denote by $\Loc(G)$ the set of all local operators on $G$. 
It is easily seen that $\Loc(G)$ is a unital associative algebra of continuous linear operators 
on the function space~$\Ec(G)$. 
\end{definition}

\begin{remark}\label{rem_loc}
 \normalfont
It follows from \cite{Pe60} that if $G$ is any finite-dimensional Lie group, 
then $\Loc(G)$ is precisely the set of linear differential operators (possibly of infinite order) on~$G$. 
Some generalizations of that statement for locally compact groups were obtained in \cite[Th. 2.3]{Ak95a}  
and \cite[Th. 2.3]{Ed88}. 
See also \cite{WD73} and \cite{LW11} for some generalizations to the situation when $G=(\Xc,+)$ 
for Banach spaces~$\Xc$ that admit suitable bump functions (in particular for Hilbert spaces). 
\end{remark}

\begin{definition}
\normalfont
Let $G$ be any topological group and recall the notation 
$$(\forall x\in G)\quad L_x\colon G\to G,\ L_x(y)=xy.$$ 
The  \emph{left-invariant local operators} on $G$ are the elements of the set  
$$\Uc(G):=\{D\in\Loc(G)\mid (\forall x\in G)(\forall f\in\Ci(G))\quad D(f\circ L_x)=(Df)\circ L_x\}. $$
Note that $\Uc(G)$ is a unital associative subalgebra of~$\Loc(G)$. 

For every $\gamma\in\Lg(G)$ we have $D^\lambda_\gamma\in\Uc(G)$ by Remark~\ref{transl_deriv}. 
We denote by $\Uc_0(G)$ the unital associative subalgebra of $\Uc(G)$ 
generated by the family $\{D^\lambda_\gamma\mid \gamma\in\Lg(G)\}$. 
\end{definition}

We can now state one of the main problems that have motivated the present paper. 
We will address this problem in Theorem~\ref{P 4.29} and Corollary~\ref{density} below. 

\begin{problem}\label{probl1}
For any topological group $G$ we have the following inclusions of unital associative algebras: 
$$\Uc_0(G)\subseteq\Uc(G)\subseteq\Loc(G).$$
Investigate the gap between $\Uc_0(G)$ and $\Uc(G)$, 
and in particular 
find necessary or sufficient conditions on $G$ in order to ensure that $\Uc_0(G)=\Uc(G)$. 
\end{problem}

\begin{remark}\label{PBW_rem}
 \normalfont 
If $G$ is any finite-dimensional Lie group, then it follows by the Poincar\'e-Birkhoff-Witt theorem 
(see also the above Remark~\ref{rem_loc}) that $\Uc_0(G)=\Uc(G)$ and this is precisely the complexified universal enveloping algebra $U(\gg_{\CC})$ of the Lie algebra $\gg$ of~$G$. 
Hence the difficulty of Problem~\ref{probl1} 
 lies in the fact that one should extend the Poincar\'e-Birkhoff-Witt theorem from Lie groups to topological groups. 

If $G$ is a pre-Lie group (Definition~\ref{Def3.1}) with $\Lg(G)=\gg$ or a locally convex Lie group,  
then \cite[Sect. 1.3.1, ($\widetilde{4}$)]{BCR81} shows that 
the mapping $\Lg(G)\to\Uc_0(G)$, $\gamma\mapsto D^\lambda_\gamma$ 
extends to a natural surjective homomorphism of unital associative algebras $U(\gg_{\CC})\to\Uc_0(G)$,  
whose injectivity can be established   
under the additional assumption that $\Ci(G)$ contains sufficiently many functions, in some sense. 
(See the method of proof of the Poincar\'e-Birkhoff-Witt theorem from \cite{CW99} 
and also \cite[Cor. 4.1.1.7]{BCR81}.)

For instance, assume that there exists some smooth function on $\gg$ 
which is equal to~$1$ on some neighborhood of $0\in\gg$ 
and has bounded support. 
Then 
the aforementioned natural homomorphism is an isomorphism $U(\gg_{\CC})\mathop{\rightarrow}\limits^\sim \Uc_0(G)$ 
for any locally exponential Lie group~$G$ (in the sense of \cite[Sect. IV]{Ne06}) 
whose Lie algebra is $\gg$, which include in particular the Banach-Lie groups. 
Note however that this method is not always applicable 
since there exist Banach spaces that do not admit any nontrivial smooth function with bounded support, 
for instance $\ell^1(\NN)$; 
see \cite[Sect. 2, Ex. (i)]{BF66}. 

On the other hand, a result of the type $U(\gg_{\CC})\mathop{\rightarrow}\limits^\sim \Uc_0(G)=\Uc(G)$ 
was obtained in \cite[Cor. 2.5]{Ak95a} 
in the case when $G$ is any locally compact group, 
by using the Lie algebra $\gg=\Lg(G)$ discovered in \cite{Glu57} and \cite{La57}. 
\end{remark}

\section{Exponential law for smooth functions on topological groups}\label{Sect4}

The main result of this section is Theorem~\ref{P 4.15}, 
which provides a kind of exponential law for smooth functions on topological groups. 
See for instance \cite[Ch. I, \S 3]{KM97} for a broad discussion on 
the exponential law for smooth functions on open subsets of locally convex spaces. 
Further information and references on this topic can be found in \cite{KMR14},  \cite{Gl13}, and \cite{Al13}. 

\begin{notation}\label{N5.2}
\normalfont 
Let $G$ and $H$ be arbitrary topological groups. 
For any locally convex space $\Yc$ and any function 
$\varphi\in\Ci(G\times H,\Yc)$ we define  
$$\widetilde{\varphi}\colon G\to \Ci(H,\Yc),\quad \widetilde{\varphi}(x)(y)=\varphi(x,y).$$
This notation will be preserved throughout the present section. 
\end{notation}


\subsection*{Some basic formulas on partial derivatives}
We now give a definition whose correctness is established in Lemma~\ref{L 4.8} below. 

\begin{definition}\label{def 4.3}
\normalfont
Let $\varphi\in\Ci(H\times G,\Yc)$.
For $n\ge 1$ we define the partial derivatives 
$(D^\lambda _1)^n\varphi\colon H\times G\times\Lg^n(H)\to \Yc$ 
and 
$(D^\lambda _2)^n\varphi\colon H\times G\times\Lg^n(G)\to \Yc$ thus: 

For $n=1$, $\beta\in\Lg(H)$, $\alpha\in\Lg(G)$,  
$$\begin{aligned}
(D^\lambda _1 \varphi)(x,g;\beta)
&=
\frac{\de}{\de t}\Big\vert_{t=0} \varphi(x \beta(t),g) , \\
(D^\lambda _2 \varphi)(x,g;\alpha)
&=
\frac{\de}{\de t}\Big\vert_{t=0} \varphi(x,g \alpha(t)).
\end{aligned}$$
Furthermore, we define inductively 
\allowdisplaybreaks
\begin{align}
((D^\lambda _1)^{n+1}\varphi)
&(x,g;\beta_1,\dotsc,\beta_n,\beta_{n+1}) \nonumber \\
&=
\frac{\de}{\de t}\Big\vert_{t=0} ((D^\lambda _1)^n\varphi)
(x \beta_{n+1}(t),g;\beta_1,\dotsc,\beta_n) \nonumber\\
((D^\lambda _2)^{n+1}\varphi) 
&(x,g;\alpha_1,\dotsc,\alpha_n,\alpha_{n+1}) \nonumber\\
&=
\frac{\de}{\de t}\Big\vert_{t=0} ((D^\lambda _2)^n\varphi) 
(x,g \alpha_{n+1}(t);\alpha_1,\dotsc,\alpha_n). \nonumber
\end{align}
\end{definition}

\begin{notation}\label{not 4.6}
\normalfont
By $\1\in \Lg(G)$ we denote the constant function from $\RR $
to $G$ given by 
$\1(t)=\1\in G$ for all $t\in \RR $.
\end{notation}

The following lemma ensures the existence and
continuity of the maps  
$(D^\lambda _1)^n\varphi $ and $(D^\lambda _2)^n\varphi$ from Definition~\ref{def 4.3}.

\begin{lemma}\label{L 4.8}
If $\varphi \in \Ci(H\times G,\Yc)$, 
then for all $x\in H$, $g\in G$, $n\ge 1$,  
$\beta_1,\dotsc,\beta_n \in \Lg(H)$, 
$\alpha_1,\dots,\alpha_n\in\Lg(G)$, 
we have:
\begin{itemize}
\item[(a)] $((D^\lambda _1)^n\varphi)(x,g;\beta_1,\dotsc,\beta_n)=
((D^\lambda)^n\varphi)(x,g;(\beta_1,\1),(\beta_2,\1),\dotsc,(\beta_n,\1))$
\item[(b)] $((D^\lambda _2)^n\varphi)(x,g;\alpha_1,\dotsc,\alpha_n)=
((D^\lambda)^n\varphi)(x,g;(\1,\alpha_1),(\1,\alpha_2),\dotsc,(\1,\alpha_n))$
\item[(c)] The maps 
$(D^\lambda _1)^n\varphi\colon H\times G\times\Lg^n(H)\to \Yc$ 
and 
$(D^\lambda _2)^n\varphi\colon H\times G\times\Lg^n(G)\to \Yc$ are continuous.
\end{itemize}
\end{lemma}

\begin{proof}
Assertions (a) and (b) are straightforward, 
and (c) follows from (a) and (b), by using the hypothesis
$\varphi \in \Ci(H\times G,\Yc)$.
\end{proof}

\begin{proposition}\label{P 4.9}
For every $\varphi \in \Ci(H\times G,\Yc)$, $n\ge 1$, and $\beta_1,\dotsc,\beta_n \in \Lg(H)$, 
we have 
$$\begin{aligned}
((D^\lambda)^n \widetilde{\varphi})(x;\beta_1,\dotsc,\beta_n)(g)
&=((D^\lambda _1)^n\varphi)(x,g;\beta_1,\dotsc,\beta_n)  \\
&=((D^\lambda)^n\varphi)(x,g;(\beta_1,\1),(\beta_2,\1),\dotsc,(\beta_n,\1)).
\end{aligned}$$
\end{proposition}

\begin{proof}
The last equality follows from Lemma \ref{L 4.8}(a). 
For the first 
equality we will perform the proof by induction on $n$.

The case $n=1$: 
For $x\in H$, $g\in G$ and $\beta\in\Lg(H)$ we have 
$$(D^\lambda \widetilde{\varphi})(x;\beta)(g)=\frac{\de}{\de t}\Big\vert_{t=0} 
\widetilde{\varphi}(x \beta(t))(g)
=\frac{\de}{\de t}\Big\vert_{t=0} \varphi(x \beta(t),g)
=(D^\lambda _1 \varphi)(x,g;\beta).$$
Now suppose that the assertion was already proved  for $n$. 
For $n+1$ we have 
$$\begin{aligned}
((D^\lambda )^{n+1}\widetilde{\varphi})(x;\beta_1,\dotsc,\beta_n,\beta_{n+1})(g)
&=\frac{\de}{\de t}\Big\vert_{t=0} ((D^\lambda)^n \widetilde{\varphi})
(x \beta_{n+1}(t);\beta_1,\dotsc,\beta_n)(g) \\
&=\frac{\de}{\de t}\Big\vert_{t=0} ((D^\lambda _1)^n\varphi)
(x \beta_{n+1}(t),g;\beta_1,\dotsc,\beta_n) \\
&=((D^\lambda _1)^{n+1}\varphi)(x,g;\beta_1,\dotsc,\beta_n,\beta_{n+1}) 
\end{aligned}$$
and proof ends.
\end{proof}

\begin{remark}\label{obs 4.10}
\normalfont
The formula from Proposition~\ref{P 4.9} gives us the point values  of the 
derivatives of the function $\widetilde{\varphi}$ introduced in Notation~\ref{N5.2}. 
We still have to show that the convergence of the differential quotients to these values 
holds in the topology of $\Ec(G,\Yc)$. 
This task will be accomplished in Proposition~\ref{P 4.27}. 
\end{remark}

\begin{lemma}\label{L 4.11}
If $\varphi\in\Ci(H\times G,\Yc)$, then for all $x\in H$, $g\in G$, $n\ge 1$, and $\alpha_1,\dots,\alpha_n\in\Lg(G)$ we have 
$$\begin{aligned}
((D^\lambda )^n(\widetilde{\varphi}(x)))(g;\alpha_1,\dotsc,\alpha_n)
&=
((D^\lambda _2)^n\varphi)(x,g;\alpha_1,\dotsc,\alpha_n) \\
&=((D^\lambda)^n\varphi)(x,g;(\1,\alpha_1),(\1,\alpha_2),\dotsc,(\1,\alpha_n)).
\end{aligned}$$
\end{lemma}

\begin{proof}
The proof is similar to the one of Proposition~\ref{P 4.9}. 
\end{proof}

\begin{remark}\label{obs 4.12}
\normalfont
It follows by Lemma \ref{L 4.11} and Lemma \ref{L 4.8}(c) that 
if $\varphi\in\Ci(H\times G,\Yc)$, then 
for all $x\in G$ we have 
$\widetilde{\varphi}(x):=\varphi(x,\cdot)\in \Ci(G,\Yc)$, hence the function
$\widetilde{\varphi}\colon H\to \Ec(G,\Yc)$ (see Notation~\ref{N5.2}) is well defined.
\end{remark}

\subsection*{Continuity of $\widetilde{\varphi}$} 

\begin{proposition}\label{P 4.14}
If $\varphi\in\Ci(H\times G,\Yc)$, then 
the function $\widetilde{\varphi}\colon H\to \Ec(G,\Yc)$ (see Notation~\ref{N5.2}) 
is continuous. 
\end{proposition}

\begin{proof}
We will show that for any  seminorm $p=p_{{K_1},{K_2}}$ on $\Ec(G,\Yc)$ as in Definition \ref{distrib}
we have $\lim\limits_{x\to x_0} p(\widetilde{\varphi}(x)-\widetilde{\varphi}(x_0))=0$, 
which is tantamount to the following condition: 
$$(\forall x_0\in H)(\forall \varepsilon>0)(\exists U\in \mathcal{V}(x_0))(\forall x\in U)\quad 
p(\widetilde{\varphi}(x)-\widetilde{\varphi}(x_0))\le \varepsilon.$$
According to the compact sets $K_1$ and $K_2$ involved in the definition of
 the seminorm~$p$, we will analyze separately the two cases that can occur. 
Let $x_0\in H$ and $\varepsilon>0$ be arbitrary, fixed throughout the proof. 

Case (a): $p=p_{{K_1},{K_2}}$, where $K_2\subseteq G$ is any compact set and $K_1=\emptyset$.
If we denote $E(x):=p(\widetilde{\varphi}(x)-\widetilde{\varphi}(x_0))$, 
then  
$$E(x)=\sup_{g\in {K_2}} \lvert \widetilde{\varphi}(x)(g)-\widetilde{\varphi}(x_0)(g) \rvert
=\sup_{g\in {K_2}} \lvert \varphi(x,g)-\varphi(x_0,g) \rvert$$
hence the conclusion follows directly by applying Lemma \ref{L 4.2} with 
$x_0\in X=H$, $T=G$, $K=K_2$ and $f=\varphi\colon H\times G\to \Yc$,  
which is a continuous function since $\varphi \in \Ci(H\times G,\Yc)$. 
This completes the proof of Case (a). 

Case (b)  $p=p_{{K_1},{K_2}}$, for arbitrary compact sets $K_2\subseteq G$ and  
 $K_1\subseteq \Lg^n(G)$ compact, for some $n\ge 1$.

Denote again $E(x):=p(\widetilde{\varphi}(x)-\widetilde{\varphi}(x_0))$. 
In this case we have 
$$\begin{aligned}
E(x)
&=\sup\{ \lvert ((D^\lambda)^n (\widetilde{\varphi}(x)-\widetilde{\varphi}(x_0)))(g;\gamma)\rvert \mid 
g\in K_2,\gamma\in K_1 \} \\
&=\sup\{\lvert ((D^\lambda)^n(\widetilde{\varphi}(x)))(g;\gamma)-
((D^\lambda)^n(\widetilde{\varphi}(x_0)))(g;\gamma) \rvert \mid 
g\in K_2,\gamma\in K_1 \}.
\end{aligned}$$
By using Lemma \ref{L 4.11} we obtain  
$$E(x)=\sup\{\lvert ((D^\lambda_2)^n \varphi) (x,g;\gamma)- 
((D^\lambda_2)^n \varphi) (x_0,g;\gamma) \rvert \mid  
g\in K_2,\gamma\in K_1  \}$$
hence the conclusion follows by applying Lemma \ref{L 4.2} with 
$x_0\in X=H$, $T=G\times \Lg^n(G)$, and the compact set $K=K_2\times K_1\subseteq T$, 
since 
$f=(D^\lambda_2)^n \varphi\colon H\times G\times \Lg^n(G)\to \Yc$ 
is a continuous function by Lemma \ref{L 4.8}(c). 

This completes the proof.
\end{proof}

\subsection*{Smoothness of $\widetilde{\varphi}$} 

\begin{definition}\label{def 4.16}
\normalfont
Let $G$ be any topological group. 
We say that $\alpha,\beta \in \Lg(G)$ \emph{commute} if 
$$(\forall s,t\in\RR )\quad \alpha(t)\beta(s)=\beta(s)\alpha(t).$$
\end{definition}

\begin{remark}\label{obs 4.17}
\normalfont
If $G$ and $H$ are any topological groups, then 
every $\alpha \in \Lg(G)$ commutes with $\1\in \Lg(G)$, 
and for every  $\alpha\in\Lg(G)$ and $\beta \in \Lg(H)$ the elements
 $(\1,\alpha)$ and $(\beta,\1)$ from $\Lg(H\times G)$ commute.
\end{remark} 

\begin{lemma}\label{L 4.18}
Let $H$ be any topological group, $n\ge 2$, and
$\gamma_1,\dots,\gamma_n\in \Lg(H)$. 
Assume that 
$\gamma_i$ 
commutes with $\gamma_{i+1}$ for some $i\in \{1,2,\dots,n-1 \}$.
Then for any $f\in\Cc^n(H,\Yc)$ and $x\in H$ we have
$$\begin{aligned}
((D^\lambda)^n f)&(x;\gamma_n,\dotsc,\gamma_{i+2}, 
\gamma_{i+1},\gamma_ i,\gamma_{i-1},\dotsc,\gamma_1) \\
&=((D^\lambda)^n f)(x;\gamma_n,\dotsc,\gamma_{i+2}, 
\gamma_i,\gamma_ {i+1},\gamma_{i-1},\dotsc,\gamma_1).
\end{aligned}$$
\end{lemma}

\begin{proof}

The function 
$$(t_1,\dotsc,t_n)\mapsto f(x \gamma_1(t_1)\gamma_2(t_2)\cdots \gamma_n(t_n))$$
belongs to $\Cc^n(\RR^n,\Yc)$ by 
\cite[Prop. 1.2.2.1]{BCR81}.
Therefore we obtain 
\allowdisplaybreaks
\begin{align}
((D^\lambda)^n f)&(x;\gamma_n,\dotsc,\gamma_{i+2}, 
\gamma_{i+1},\gamma_ i,\gamma_{i-1},\dotsc,\gamma_1) 
\nonumber \\
=&\frac{\partial^n}{\partial t_1\cdots\partial t_{i-1}
\partial t_i\partial t_{i+1}\partial t_{i+2}\dotsc\partial t_n}
\Big\vert_{t_1=\cdots=t_n=0}
f(x \gamma_1(t_1) \cdots \gamma_n(t_n)) \nonumber \\
=&\frac{\partial^n}{\partial t_1\cdots\partial t_{i-1}
\partial t_{i+1}\partial t_i \partial t_{i+2}\cdots\partial t_n}
\Big\vert_{t_1=\cdots=t_n=0}
f(x \gamma_1(t_1) \cdots \gamma_n(t_n)) \nonumber \\
=&\frac{\partial^n}{\partial t_1\cdots\partial t_{i-1}
\partial t_{i+1}\partial t_i \partial t_{i+2}\cdots\partial t_n}
\Big\vert_{t_1=\cdots=t_n=0} \nonumber \\
&f(x \gamma_1(t_1)\cdots \gamma_{i-1}(t_{i-1})\gamma_{i+1}(t_{i+1})
\gamma_i(t_i)\gamma_{i+2}(t_{i+2})\cdots\gamma_n(t_n))
\nonumber \\
=&((D^\lambda)^n f)(x;\gamma_n,\dots,\gamma_{i+2}, 
\gamma_i,\gamma_ {i+1},\gamma_{i-1},\dots,\gamma_1) \nonumber
\end{align}
which is just the required relationship.
\end{proof}

\begin{lemma}\label{L 4.19}
Let $G$ and $H$ be any topological groups and $\varphi \in \Ci(H\times G,\Yc)$. 
For some $s\ge 1$ let $\alpha_1,\dotsc,\alpha_s\in \Lg(G)$. 
Then the following assertions hold for all $x\in H$ and $g\in G$: 
\begin{itemize}
\item[(a)] For every   $\beta\in \Lg(H)$ we have  
$$((D^\lambda)^{s+1} \varphi)(x,g;(\beta,\1),(\1,\alpha_1),\dotsc,(\1,\alpha_s)) 
\hskip-0.8pt
=
\hskip-0.8pt
((D^\lambda)^{s+1} \varphi)(x,g;(\1,\alpha_1),\dotsc,(\1,\alpha_s),(\beta,\1)).$$
\item[(b)] For every $n\ge 1$ and $\beta_1,\dotsc,\beta_n,\beta_{n+1}\in \Lg(H)$ we have 
$$\begin{aligned}
((& D^\lambda)^{n+s+1}\varphi)(x,g;(\beta_1,\1),\dotsc,(\beta_n,\1),(\beta_{n+1},\1),
(\1,\alpha_1),\dotsc,(\1,\alpha_s)) \\
&=((D^\lambda)^{n+s+1}\varphi)(x,g;(\beta_1,\1),\dotsc,(\beta_n,\1),
(\1,\alpha_1),\dotsc,(\1,\alpha_s),(\beta_{n+1},\1)).
\end{aligned}$$
\end{itemize}
\end{lemma}

\begin{proof}
In both assertions one can start 
from the right-hand side of the equality to be proved, 
and one uses   
Remark \ref{obs 4.17} and Lemma \ref{L 4.18} for $H$ replaced by $H\times G$ 
for the pairs 
$(\1,\alpha_i)$ and $(\beta_{n+1},\1)$. 
One can thus obtain the order of arguments in the left-hand side of each of the two desired equalities. 
\end{proof}


We are now in a position to solve the problem mentioned in Remark~\ref{obs 4.10}.

\begin{proposition}\label{P 4.27}
Let $G$ and $H$ be any topological groups, $\Yc$ be any locally convex space and 
for any $\varphi\in\Ci(H\times G,\Yc)$ define as above 
$\widetilde{\varphi}\colon H\to \Ec(G,\Yc)$, $\widetilde{\varphi}(x)(g)=\varphi(x,g)$. 
Then for every 
$x_0\in G$ and $\beta^0_1,\dotsc,\beta^0_{n+1}\in \Lg(H)$ we have
$$\begin{aligned}
\lim\limits_{t\to 0}&\frac
{((D^\lambda)^n \widetilde{\varphi})(x_0 \beta^0_{n+1}(t);\beta^0_1,\dotsc,\beta^0_n)
-((D^\lambda)^n \widetilde{\varphi})(x_0;\beta^0_1,\dotsc,\beta^0_n)}
{t}\\ 
&= ((D^\lambda _1)^{n+1}\varphi)(x_0,\bullet\,;\beta^0_1,\dotsc,\beta^0_n,\beta^0_{n+1}) 
\end{aligned}$$
in the topology of $\Ec(G,\Yc)$ from Definition \ref{distrib}. 
\end{proposition}

\begin{proof}
Define $h\colon \RR\to \Ec(G,\Yc)$ by  
$$h(t)=
\begin{cases}
\frac
{((D^\lambda)^n \widetilde{\varphi})(x_0 \beta^0_{n+1}(t);\beta^0_1,\dotsc,\beta^0_n)
-((D^\lambda)^n \widetilde{\varphi})(x_0;\beta^0_1,\dotsc,\beta^0_n)}
{t} &\text{ if  }t\neq 0, \\  
((D^\lambda _1)^{n+1}\varphi)(x_0,\bullet\, ;\beta^0_1,\dotsc,\beta^0_n,\beta^0_{n+1}) 
&\text{ if }t=0.
\end{cases}$$
We must prove that $\lim\limits_{t\to 0}h(t)=h(0)$ in $\Ec(G,\Yc)$, 
that is, for every seminorm 
$p=p_{{K_1},{K_2}}$ (see Definition~\ref{distrib}) we have 
$\lim\limits_{t\to 0} p(h(t)-h(0))=0$.

Depending on the seminorm $p$, we distinguish two cases that can occur. 

Case 1: $p=p_{{K_1},{K_2}}$ for an arbitrary compact set $K_2\subseteq G$ and $K_1=\emptyset$.

We denote  $E(t)=p(h(t)-h(0)) $ and then we have 
$$\begin{aligned}
E(t)
=&\sup\{\lvert h(t)(g)-h(0)(g) \rvert \mid g\in K_2 \} \\ 
=&\sup\{\lvert \frac
{((D^\lambda_1)^n \varphi)(x_0 \beta^0_{n+1}(t),g;\beta^0_1,\dotsc,\beta^0_n) 
-((D^\lambda_1)^n \varphi)(x_0,g;\beta^0_1,\dotsc,\beta^0_n)}
{t} \\
&-((D^\lambda _1)^{n+1}\varphi)(x_0,g;\beta^0_1,\dotsc,\beta^0_n,\beta^0_{n+1})\rvert \mid g\in K_2 \} \\
=&\sup\{\lvert F(t,g)-F(0,g) \rvert \mid g\in K_2 \}
\end{aligned}$$
where 
$F\colon \RR\times G\to\Yc$ is defined by 
$$F(t,g)=
\begin{cases}
\frac
{((D^\lambda_1)^n \varphi)(x_0 \beta^0_{n+1}(t),g;\beta^0_1,\dotsc,\beta^0_n)
-((D^\lambda_1)^n \varphi)(x_0,g;\beta^0_1,\dotsc,\beta^0_n)}
{t}  &\text{ if }t\neq 0, \\ 
((D^\lambda _1)^{n+1}\varphi)(x_0,g;\beta^0_1,\dotsc,\beta^0_n,\beta^0_{n+1}) 
&\text{ if }t=0.
\end{cases}$$
The desired property $\lim\limits_{t\to 0}E(t)=0$ 
will follow by an application of Lemma \ref{L 4.2} for $X=\RR $, $T=G$,
$x_0=0\in \RR$, $K=K_2\subseteq G$ and 
$f=F\colon \RR\times G\to\Yc$, 
as soon as we will have checked that $F$ is a continuous function. 

To this end, first note that for arbitrary $g\in G$ we have  
$$\begin{aligned}
\lim\limits_{t\to 0}F(t,g)
&=\frac{\de}{\de t}\Big\vert_{t=0} 
((D^\lambda_1)^n\varphi)(x_0\beta^0_{n+1}(t),g;\beta^0_1,\dotsc,\beta^0_n) \\
&=((D^\lambda _1)^{n+1}\varphi)(x_0,g;\beta^0_1,\dotsc,\beta^0_n,\beta^0_{n+1})=F(0,g). 
\end{aligned}$$
Next, we will show that 
Lemma~\ref{L 4.26} applies for $H$ replaced by $H\times G$, $(x_0,g)\in H\times G$,
$X=(\beta^0_{n+1},\1)\in \Lg(H\times G)$, and 
$f\colon H\times G\to \Yc $, 
$f(x,y)=((D^\lambda _1)^n\varphi)(x,y;\beta^0_1,\dotsc,\beta^0_n) $, 
which is continuous since 
 $\varphi \in \Ci(H\times G,\Yc)$.
Note that the derivative 
$D^{\lambda}_X f\colon H\times G\to \Yc$ 
is given by 
$(D^{\lambda}_X f)(x,y)=
((D^\lambda _1)^{n+1}\varphi)(x,y;\beta^0_1,\dotsc,\beta^0_n,\beta^0_{n+1}) $, 
and this derivative is a continuous function  since $\varphi \in \Ci(H\times G,\Yc)$.

Therefore Lemma~\ref{L 4.26} applies and provides a continuous function 
$\chi\colon \RR \times G\to\Yc$ satisfying for arbitrary $g\in G$ the conditions 
$\chi(0,g)=0$ and 
$f(x_0 \beta^0_{n+1}(t),g)=f(x_0,g)+t(D^{\lambda}_X f)(x_0,g)+ t \chi(t,g)$.
We have 
$$\begin{aligned}
(D^{\lambda}_X f)(x_0,g)
&=(D^{\lambda}f)(x_0,g;\beta^0_{n+1},1)
=\frac{\de}{\de t}\Big\vert_{t=0} f(x_0 \beta^0_{n+1}(t),g) \\
&=\frac{\de}{\de t}\Big\vert_{t=0} ((D^\lambda_1)^n\varphi)(x_0\beta^0_{n+1}(t),g;\beta^0_1,\dotsc,\beta^0_n) \\
&=((D^\lambda _1)^{n+1}\varphi)(x_0,g;\beta^0_1,\dotsc,\beta^0_n,\beta^0_{n+1})\\
&=F(0,g). 
\end{aligned}$$
and $F(t,g)=F(0,g)+\chi(t,g)$, 
hence $F$ is the sum of two continuous functions, 
since $\chi$ is continuous by Lemma~\ref{L 4.26} and 
$g\mapsto F(0,g)=((D^\lambda _1)^{n+1}\varphi)(x_0,g;\beta^0_1,\dotsc,\beta^0_n,\beta^0_{n+1}) $ 
is a continuous function since $\varphi \in \Ci(H\times G,\Yc)$. 
Therefore $F$ itself is continuous, and this concludes the analysis of Case~1. 

Case 2: $p=p_{{K_1},{K_2}}$ for arbitrary compact sets $K_2\subseteq G$ and   
 $K_1\subseteq \Lg^s(G)$, where $s\ge 1$.
Denote again $E(t)=p(h(t)-h(0))$ for $t\in\RR$. 
Then we have 
$$E(t)=\sup\limits_{g,\alpha}\lvert((D^\lambda)^s(h(t))(g;\alpha_1,\dotsc,\alpha_s)-  
((D^\lambda)^s(h(0))(g;\alpha_1,\dotsc,\alpha_s) \rvert$$
where 
$g\in K_2$, $\alpha=(\alpha_1,\dotsc,\alpha_s)\in K_1$. 
It then follows that $E(t)$ is the supremum of the values of the seminorm $\vert\cdot\vert$ 
involved in the definition of $p=p_{K_1,K_2}$ 
(see Definition~\ref{distrib})
on the vectors in $\Yc$ 
of the form 
\allowdisplaybreaks
\begin{align}
\frac{1}{t}& \Bigl(((D^\lambda)^{n+s} \varphi)
(x_0 \beta^0_{n+1}(t),g;
(\beta^0_1,\1),\dotsc,(\beta^0_n,\1),(\1,\alpha_1),\dotsc,(\1,\alpha_s)) 
\nonumber \\
&-
((D^\lambda)^{n+s} \varphi)
(x_0,g;(\beta^0_1,\1),\dotsc,(\beta^0_n,\1),(\1,\alpha_1),\dotsc,(\1,\alpha_s))\Bigr) 
\nonumber \\
& - 
((D^\lambda)^{n+s+1} \varphi)
(x_0,g;(\beta^0_1,\1),\dotsc,(\beta^0_n,\1),(\beta^0_{n+1},\1),(1,\alpha_1),\dotsc,(\1,\alpha_s)) \nonumber
\end{align}
where again 
$g\in K_2$ and $\alpha=(\alpha_1,\dotsc,\alpha_s)\in K_1$. 

Therefore 
$E(t)=\sup\{\lvert F(t,g,\alpha)- F(0,g,\alpha) \rvert \mid 
g\in K_2,\alpha \in K_1 \}$
where  the function 
$F\colon \RR\times G\times \Lg^s(G)\to \Yc$  is given by 
$$
F(t,g,\alpha)
=
\begin{cases}
\begin{aligned}
& 
\frac{1}{t} \Bigl(((D^\lambda)^{n+s} \varphi)
(x_0 \beta^0_{n+1}(t),g;
(\beta^0_1,1),\dotsc,(\beta^0_n,\1),(\1,\alpha_1),\dotsc,(\1,\alpha_s)) 
\nonumber \\
&-
((D^\lambda)^{n+s} \varphi)
(x_0,g;(\beta^0_1,\1),\dotsc,(\beta^0_n,\1),(\1,\alpha_1),\dotsc,(\1,\alpha_s))\Bigr), 
\end{aligned}
& \\
\hfill \text{ if }t\ne0, & \\
((D^\lambda)^{n+s+1} \varphi)
(x_0,g;(\beta^0_1,\1),\dotsc,(\beta^0_n,\1),(\beta^0_{n+1},\1),(\1,\alpha_1),\dotsc,(\1,\alpha_s)), 
& \\
\hfill \text{ if }t=0. & \\ 
\end{cases}
$$
The desired property $\lim\limits_{t\to 0}E(t)=0$ then follows by an 
application of Lemma \ref{L 4.2} for 
$X=\RR $, $T=G\times \Lg^s(G)$, $x_0=0\in \RR $,
the compact $K=K_2\times K_1\subseteq G\times \Lg^s(G)$ and the function $f=F$, 
as soon as we will have proved that $F$ is a continuous function. 

Just as in Case~1, we first note that 
$$\begin{aligned}
\lim\limits_{t\to 0}F(t,g,\alpha)
&=\frac{\de}{\de t}\Big\vert_{t=0} ((D^\lambda)^{n+s} \varphi)
(x_0 \beta^0_{n+1}(t),g;
(\beta^0_1,\1),\dotsc,(\beta^0_n,\1),(\1,\alpha_1),\dotsc,(\1,\alpha_s)) \\
&=((D^\lambda)^{n+s+1} \varphi)
(x_0,g;(\beta^0_1,\1),\dotsc,(\beta^0_n,\1),(\1,\alpha_1),\dotsc,(\1,\alpha_s),(\beta^0_{n+1},\1)) \\
&=F(0,g,\alpha) 
\end{aligned}$$ 
by using Lemma~\ref{L 4.19}(b).

Now let 
$$B\colon \RR\to\Yc,\quad  
B(t)=((D^\lambda)^{n+s} \varphi)
(x_0 \beta^0_{n+1}(t),g;
(\beta^0_1,\1),\dotsc,(\beta^0_n,\1),(\1,\alpha_1),\dotsc,(\1,\alpha_s)) .$$
Since $\varphi \in \Ci(H\times G,\Yc)$, we have 
$B\in \Cc^1(\RR,\Yc)$, 
$B'(0)=F(0,g,\alpha)$ (Lemma \ref{L 4.19}(b)) and 
 $B'(t)=((D^\lambda)^{n+s+1} \varphi)
(x_0 \beta^0_{n+1}(t),g;
(\beta^0_1,\1),\dotsc,(\beta^0_n,\1),(\1,\alpha_1),\dotsc,(\1,\alpha_s),(\beta^0_{n+1},\1))  $.

We have by the fundamental theorem of calculus for functions with values in the space~$\Yc$ which may not be complete 
(see \cite[Th. 1.5]{Gl02a}) 
$$B(t)=B(0)+t\int\limits_0^1 B'(tz)\de z
=B(0)+tB'(0)+t\int\limits_0^1 B'(tz)\de z-tB'(0)$$
and therefore 
\begin{equation}\label{P 4.21_proof_eq1}
F(t,g,\alpha)=F(0,g,\alpha)+ \chi(g,t,\alpha)
\end{equation}
where
$\chi\colon G\times\RR\times\Lg^s(G)\to\Yc$ 
is given by 
$$\begin{aligned}
\chi(g,& t,\alpha) \\
=& \int\limits_0^1 
((D^\lambda)^{n+s+1}\varphi)(x_0 \beta^0_{n+1}(tz),g;
(\beta^0_1,\1),\dotsc,(\beta^0_n,\1),(\1,\alpha_1),\dotsc,(\1,\alpha_s),(\beta^0_{n+1},\1))\de z \\
&-((D^\lambda)^{n+s+1}\varphi)(x_0,g;
(\beta^0_1,\1),\dotsc,(\beta^0_n,\1),(\1,\alpha_1),\dotsc,(\1,\alpha_s),(\beta^0_{n+1},\1)).
\end{aligned}$$
We have $\chi(g,0,\alpha)=0 $ and $\chi$ is continuous by Lemma~\ref{T 4.25} 
applied for 
$X=G\times\RR \times\Lg^s(G)$ and 
$f\colon G\times\RR\times\Lg^s(G)\times [0,1]\to\Yc$ 
given by 
$$\begin{aligned}
f(g,t,& \alpha,z) \\
=&((D^\lambda)^{n+s+1}\varphi)(x_0 \beta^0_{n+1}(tz),g;
(\beta^0_1,\1),\dotsc,(\beta^0_n,\1),(\1,\alpha_1),\dotsc,(\1,\alpha_s),(\beta^0_{n+1},\1))
\end{aligned} $$
which is continuous since $\varphi \in \Ci(H\times G,\Yc)$.

Finally, by using the above equality~\eqref{P 4.21_proof_eq1}, 
we again obtain that $F$ is the sum of two continuous functions 
hence is itself continuous, and this completes the proof. 
\end{proof}

\begin{lemma}\label{L 4.13}
If $\varphi\in\Ci(H\times G,\Yc)$, then the following assertions hold. 
\begin{itemize}
\item[(a)] Let $x\in H$ and $\beta_1,\dotsc,\beta_n \in \Lg(H)$ be fixed.
Then the function
$$h:=((D^\lambda _1)^n\varphi)(x,\bullet\, ;\beta_1,\dots,\beta_n)\colon G\to \Yc$$
belongs to $\Ci(G,\Yc)$ and for all $s\ge 1$ and $\alpha_1,\dots,\alpha_s\in\Lg(G)$ we have
$$((D^\lambda)^s h)(g;\alpha_1,\dots,\alpha_s)
=((D^\lambda)^{n+s} \varphi)
(x,g;(\beta_1,\1),\dotsc,(\beta_n,\1),(\1,\alpha_1),\dots,(\1,\alpha_s)).$$

\item[(b)] Let $x\in H$, $\beta_1,\dotsc,\beta_n \in \Lg(H)$, and $\gamma_1,\dots,\gamma_n \in \Lg(G)$ be fixed. 
Then the function 
$$h:=((D^\lambda)^n\varphi)(x,\bullet\,;(\beta_1,\gamma_1),\dots,(\beta_n,\gamma_n))\colon G\to\Yc$$
is in $\Ci(G,\Yc)$ and for every $s\ge1$ and $\alpha_1,\dots,\alpha_s\in\Lg(G)$ we have
$$((D^\lambda)^s h)(g;\alpha_1,\dots,\alpha_s)
=((D^\lambda)^{n+s} \varphi)
(x,g;(\beta_1,\gamma_1),\dots,(\beta_n,\gamma_n),(\1,\alpha_1),\dots,(\1,\alpha_s)).$$
\end{itemize}
\end{lemma}

\begin{proof}
Assertion~(a) follows by Assertion~(b) for
$\gamma_1=\cdots=\gamma_n=\1 \in \Lg(G)$, by using Lemma~\ref{L 4.8}(a).

Assertion~(b) will be proved by induction on $s\ge 1$.
Since $\varphi \in \Ci(H\times G,\Yc)$ and $h(g)=((D^\lambda)^n\varphi)(x,g;(\beta_1,\gamma_1),\dots,(\beta_n,\gamma_n))$
it follows that the function $h$ is 
continuous.

The case $s=1$: 
We have 
$$\begin{aligned}
(D^\lambda h)(g;\alpha)
&=\frac{\de}{\de t}\Big\vert_{t=0}h(g \alpha(t)) \\
&=\frac{\de}{\de t}\Big\vert_{t=0}
((D^\lambda)^n\varphi)(x,g\alpha(t);(\beta_1,\gamma_1),\dots,(\beta_n,\gamma_n)) \\
&=((D^\lambda)^{n+1}\varphi)(x,g;(\beta_1,\gamma_1),\dots,(\beta_n,\gamma_n),(\1,\alpha)) 
\end{aligned}$$
and then since  $\varphi \in \Ci(H\times G,\Yc)$ we obtain $h\in \Cc^1(G,\Yc)$.

Now suppose the assertion was proved for $s$ and we will prove it 
for $s+1$.
We have 
$$\begin{aligned}
((D^\lambda)^{s+1}h)&(g;\alpha_1,\dotsc,\alpha_s,\alpha_{s+1}) \\
&=
\frac{\de}{\de t}\Big\vert_{t=0}((D^\lambda)^s h)(g \alpha_{s+1}(t);\alpha_1,\dotsc,\alpha_s) \\
&=\frac{\de}{\de t}\Big\vert_{t=0}((D^\lambda)^{n+s} \varphi)
(x,g\alpha_{s+1}(t);(\beta_1,\gamma_1),\dots,(\beta_n,\gamma_n),
(\1,\alpha_1),\dotsc,(\1,\alpha_s)) \\
&=((D^\lambda)^{n+s+1} \varphi)(x,g;(\beta_1,\gamma_1),\dots,(\beta_n,\gamma_n),
(\1,\alpha_1),\dotsc,(\1,\alpha_s),(\1,\alpha_{s+1}))
\end{aligned}$$
and the proof by induction ends.

Moreover, since $\varphi \in \Ci(H\times G,\Yc)$, we obtain  $h\in \Cc^s(G,\Yc)$ for every $s\ge 1 $. 
This shows that $h\in \Ci(G,\Yc)$, and the proof is complete.
\end{proof}

\begin{theorem}\label{P 4.15}
Let $G$ and $H$ be any topological groups and $\Yc$ be any locally convex space. 
Then for arbitrary $\varphi\in\Ci(H\times G,\Yc)$, the corresponding function 
$$\widetilde{\varphi}\colon H\to \Ci(G,\Yc), \quad\widetilde{\varphi}(x)(g):=\varphi(x,g).$$
belongs to $\Ci(H,\Ec(G,\Yc))$. 
Moreover, the map 
$$\Ec(H\times G,\Yc)\to \Ec(H,\Ec(G,\Yc)),\quad \varphi\mapsto\widetilde{\varphi}$$
is an isomorphism of locally convex spaces. 
\end{theorem}

\begin{proof}
To prove the first assertion, let $\varphi\in\Ci(H\times G,\Yc)$ arbitrary. 
The fact that $\widetilde{\varphi}$ is continuous follows by Proposition~\ref{P 4.14}.
We will show that for every $n\ge 1$ the derivative 
$(D^\lambda)^n\widetilde{\varphi} \colon H\times\Lg^n(H)\to \Ec(G,\Yc)$
exists and is continuous.
The existence of that derivative actually follows from Proposition~\ref{P 4.27}. 
The fact that the derivative takes values in $\Ec(G,\Yc)$ is a consequence of Lemma~\ref{L 4.13}(a).

For continuity of the above derivative we will prove that for every
seminorm $p=p_{{K_1},{K_2}}$ on $\Ec(G,\Yc)$ as in Definition~\ref{distrib} and every 
$x_0\in H$, $\beta^0_1,\dotsc,\beta^0_n \in \Lg(H)$ 
and arbitrary $\varepsilon>0$ there exists a neighborhood $U$ of 
$(x;\beta^0_1,\dotsc,\beta^0_n)\in H\times\Lg^n(H)$ 
for which for every 
$(x;\beta_1,\dotsc,\beta_n)\in U$ we have 
$p(((D^\lambda)^n \widetilde{\varphi})(x;\beta_1,\dotsc,\beta_n)-
((D^\lambda)^n \widetilde{\varphi})(x_0;\beta^0_1,\dotsc,\beta^0_n))\le \varepsilon.$

Case (a): $p=p_{{K_1},{K_2}}$, where the compact $K_2\subseteq G$ is arbitrary and $K_1=\emptyset$.

As in the proof of Proposition \ref{P 4.14}, we denote 
$$E(x;\beta_1,\dotsc,\beta_n):=\sup_{g\in K_2}
\lvert((D^\lambda)^n \widetilde{\varphi})(x;\beta_1,\dotsc,\beta_n)(g)-
((D^\lambda)^n \widetilde{\varphi})(x_0;\beta^0_1,\dotsc,\beta^0_n)(g)\rvert.$$
Applying Proposition \ref{P 4.9} we obtain 
$$E(x;\beta_1,\dotsc,\beta_n)=\sup_{g\in K_2}
\lvert((D^\lambda_1)^n\varphi)(x,g;\beta_1,\dotsc,\beta_n)-
((D^\lambda_1)^n \varphi)(x_0,g;\beta^0_1,\dotsc,\beta^0_n)\rvert.$$
Now the conclusion follows by applying Lemma \ref{L 4.2} for
$(x_0;\beta^0_1,\dotsc,\beta^0_n)\in H\times\Lg^n(H)=X$, 
$K=K_2$ compact in 
$T=G$ and 
$$f\colon  H\times\Lg^n(H)\times G\to\Yc,\quad  
f(x;\beta_1,\dotsc,\beta_n,g)=((D^\lambda_1)^n \varphi)(x,g;\beta_1,\dotsc,\beta_n),$$ 
which is a continuous function since 
$(D^\lambda_1)^n \varphi\colon H\times G\times\Lg^n(H)\to\Yc$ is
continuous by Lemma~\ref{L 4.8}(c). 

Case (b):  $p=p_{{K_1},{K_2}}$ with arbitrary compact sets $K_2\subseteq G$  and  
 $K_1\subseteq \Lg^s(G)$, where $s\ge 1$.

We denote 
$$\begin{aligned}
E(x;\beta_1,\dotsc,\beta_n)
=\sup
\{
& \lvert((D^\lambda)^s((D^\lambda)^n\widetilde{\varphi})(x;\beta_1,\dotsc,\beta_n)) (g;\gamma_1,\dotsc,\gamma_s) \\
& -((D^\lambda)^s((D^\lambda)^n\widetilde{\varphi})(x_0;\beta^0_1,\dotsc,\beta^0_n)) (g;\gamma_1,\dotsc,\gamma_s) \rvert \\
&\quad \mid g\in K_2,\gamma=(\gamma_1,\dotsc,\gamma_s) \in K_1 \}. 
\end{aligned}$$
By Proposition \ref{P 4.9} we obtain 
$$\begin{aligned}
E(x;\beta_1,\dotsc,\beta_n)
=\sup\{
& \lvert((D^\lambda)^s((D^\lambda_1)^n\varphi)(x,\bullet\,;\beta_1,\dotsc,\beta_n)) (g;\gamma_1,\dotsc,\gamma_s) \\
& -((D^\lambda)^s((D^\lambda_1)^n\varphi)(x_0,\bullet\,;\beta^0_1,\dotsc,\beta^0_n)) (g;\gamma_1,\dotsc,\gamma_s) \rvert \\
&\quad \mid g\in K_2,\gamma=(\gamma_1,\dotsc,\gamma_s) \in K_1 \}. 
\end{aligned}$$
Furthermore, by Lemma \ref{L 4.13}(a) we have 
$$\begin{aligned}
E(x;\beta_1,\dotsc,\beta_n)
=\sup\{
&\lvert((D^\lambda)^{n+s}\varphi)(x,g;(\beta_1,\1),\dotsc,(\beta_n,\1),
(\1,\gamma_1),\dotsc,(\1,\gamma_s)) \\
&-((D^\lambda)^{n+s}\varphi)(x_0,g;(\beta^0_1,\1),\dotsc,(\beta^0_n,\1),
(\1,\gamma_1),\dotsc,(\1,\gamma_s)) \rvert \\ 
&\quad\mid g\in K_2,\gamma=(\gamma_1,\dotsc,\gamma_s) \in K_1 \}. 
\end{aligned}$$
The conclusion now follows by using Lemma~\ref{L 4.2} for 
$(x_0;\beta^0_1,\dotsc,\beta^0_n)\in H\times\Lg^n(H)=X$, 
$T=G\times\Lg^s(G)$,
$K=K_2\times K_1$,  and 
$f\colon  H\times\Lg^n(H)\times G\times\Lg^s(G)\to \Yc$ given by
$$f(x,\beta_1,\dotsc,\beta_n,g,\gamma_1,\dotsc,\gamma_s)
=((D^\lambda)^{n+s}\varphi)(x,g;(\beta_1,\1),\dotsc,(\beta_n,\1),
(\1,\gamma_1),\dotsc,(\1,\gamma_s)). $$ 
Note that $f$ is continuous since 
$(D^\lambda)^{n+s}\varphi $ is continuous as a consequence of the hypothesis $\varphi \in \Ci(H\times G,\Yc)$,
and this concludes the proof of the fact that $\widetilde{\varphi}\in\Ec(H,\Ec(G,\Yc))$. 

For the second assertion, 
note that the inverse map 
$$\Ec(H,\Ec(G,\Yc))\to \Ec(H\times G,\Yc),\quad \widetilde{\varphi}\mapsto\varphi$$
is well defined, since if $\widetilde{\varphi}\in\Ec(H,\Ec(G,\Yc))$ then $\varphi\in\Ec(H\times G,\Yc)$ 
as an easy consequence of \cite[Prop. 1.2.2.3]{BCR81}. 
Moreover, the continuity of both maps $\varphi\mapsto\widetilde{\varphi}$ and $\widetilde{\varphi}\mapsto\varphi$ 
follows easily by taking into account the relations between the derivatives of $\varphi$ and $\widetilde{\varphi}$ 
provided by Proposition~\ref{P 4.9} and Lemma~\ref{L 4.8} (see also \cite[Prop. 1.2.1.5]{BCR81}). 
This completes the proof. 
\end{proof}

\begin{remark}\label{law}
\normalfont
It is easily seen that the proof of Theorem~\ref{P 4.15} 
has a local character, in the sense that it actually leads to a more general result, 
which can be stated as follows: 

Let $G$ and $H$ be any topological groups and $\Yc$ be any locally convex space. 
Pick any open sets $V\subseteq G$ and $W\subseteq H$. 
Then for arbitrary $\varphi\in\Ci(W\times V,\Yc)$, the corresponding function
the function 
$\widetilde{\varphi}\colon W\to \Ci(V,\Yc)$, $\widetilde{\varphi}(x)(g):=\varphi(x,g)$, 
belongs to $\Ci(W,\Ec(V,\Yc))$. 
Moreover, the map 
$$\Ec(W\times V,\Yc)\to \Ec(W,\Ec(V,\Yc)),\quad \varphi\mapsto\widetilde{\varphi}$$
is an isomorphism of locally convex spaces. 
\end{remark}

\section{Structure of invariant local operators}
\label{Sect5}

In this final section we establish the structure of invariant local operators on any topological group $G$ 
(Theorem~\ref{P 4.29}) and we use that result  
in order to compare to some extent the two candidates $\Uc_0(G)\subseteq\Uc(G)$ 
to the role of universal enveloping algebra of~$G$; cf.~Problem~\ref{probl1}. 
Our main result in this connection is contained in Corollary~\ref{density} below. 

\subsection*{General results}
\begin{proposition}\label{P 4.28}
If $G$ is any topological group,  
then for every 
$f\in \Ec(G) $ and  $u\in \Ec'(G) $ we have $f*u\in \Ec(G) $.
\end{proposition}

\begin{proof}
Let $m\colon G\times G\to G$, $m(x,y)=xy $.
By denoting $\check{u}=v\in \Ec'(G) $ we we have $\check{v}=u $
and $(f*u)(x)=\check{u} (f\circ L_x)=v(f\circ L_x) $.
Now  define $\varphi\colon G\times G\to {\mathbb C}$, $\varphi(x,y)=f(xy) $.
Since $\varphi=f\circ m $, it follows by Proposition~\ref{multi} that $\varphi \in \Ci(G\times G)$. 
If we define $\widetilde{\varphi}\colon G\to \Ci(G)$, $\widetilde{\varphi}(x)(y)=\varphi(x,y)$ as in Notation~\ref{N5.2}, 
then by using Theorem \ref{P 4.15} we obtain  
$\widetilde{\varphi} \in \Ci(G,\Ec(G)) $. 

Since $\widetilde{\varphi}(x)(y)=f(xy)=(f\circ L_x)(y)$, we have 
$\widetilde{\varphi}(x)=f\circ L_x$, for all $x\in G $, 
and therefore $f*u=v\circ \widetilde{\varphi} $. 
By using the property $ \widetilde{\varphi}\in\Ci(G,\Ec(G))$ provided by Theorem~\ref{P 4.15}, 
we obtain $f*u\in \Ec(G) $, and this completes the proof.
\end{proof}

We can now prove the following theorem, which extends a well-known property of finite-dimensional Lie groups. 

\begin{theorem}\label{P 4.29}
Let $G$ be any topological group 
and for every 
$u\in \Ec'(G) $ define the linear operator 
$D_u\colon \Ci(G)\to \Ci(G)$, $D_u f=f*u$

Then the operator $\Psi\colon \Ec'_{\1}(G)\to \mathcal{U}(G)$,
$\Psi(u)=D_u $ is well defined, invertible, and its inverse is 
$$\Psi^{-1}\colon \mathcal{U}(G)\to \Ec'_{\1}(G),  \quad 
(\Psi^{-1}(D))(f)=(D\check{f})(\1) 
\text{ for all }f\in \Ci(G)\text{ and }
D\in \mathcal{U}(G). $$ 
\end{theorem}

\begin{proof}
We organize the proof in three steps.

Step 1: 
We show that $\Psi $ is well defined, 
that is, for all $u\in \Ec'_1(G)$ we have $D_u\in \mathcal{U}(G)$.
In fact, 
$D_u(f\circ L_x)(y)=((f\circ L_x)*u)(y)= 
\check{u}(f\circ L_x\circ L_y) $
and on the other hand 
$(D_u(f)\circ L_x)(y)=(f*u)(xy)=\check{u}(f\circ L_{xy})=
\check{u}(f\circ L_x\circ L_y)
=D_u(f\circ L_x)(y) $, hence we obtain that 
$D_u(f\circ L_x)=D_u(f)\circ L_x $.

From $u\in \Ec'_{\1}(G) $ it follows that 
$\supp u\subseteq \{\1\}\subseteq G $, 
hence 
$$\supp(D_u f)=\supp (f\ast u)\subseteq (\supp f)(\supp u)$$ 
and therefore
 $\supp(D_u f)\subseteq (\supp f) \{\1\}=\supp f $. 
We thus obtain that $D_u\in \Uc(G) $, hence 
$\Psi $ is well defined.

Step 2: 
We show that the mapping 
$$\Phi\colon \Uc(G)\to \Ec'_{\1}(G), \quad 
(\Phi (D))(f)=(D\check{f})(\1) \text{ for all }f\in \Ci(G)\text{ and }
D\in \Uc(G) $$
is well defined, 
that is, for every  
$D\in \Uc(G) $ the functional  
 $u\colon \Ec(G)\to \CC$, 
$u(f)=(D\check{f})(\1) $, satisfies the condition $u\in\Ec'_{\1}(G)$.

To this end note that if $\supp f \subseteq U$ then  
$G\setminus U\subseteq   \{x\in G\mid f(x)=0 \} $.
Now let $x\in G$ be arbitrary with $x\neq \1 $. 
Since the topology of $G$ is assumed to be separated, 
 there exists some open neighborhood  $U$ of $x$ with $\1\notin U $.
For every  $f\in \Ci(G)$ with $\supp f \subseteq U $ 
we have $\supp \check{f}=(\supp f)^{-1} \subseteq U^{-1} $ 
and then $\supp(D \check{f})\subseteq (\supp\check{f})\subseteq U^{-1} $. 
We thus obtain $G\setminus U^{-1}\subseteq  \{y\in G\mid (D \check{f})(y) =0 \}$.

Since $\1\notin U $ and $\1\notin U^{-1} $, we have 
$(D \check{f})(\1)=0 $, hence 
$x\notin \supp u $ for arbitrary $x\in G\setminus\{\1\}$, 
and then $\supp u\subseteq \{\1\} $.
That is, $u\in \Ec'_{\1}(G) $.

Step 3:  
We show that $\Psi\circ  \Phi=\id_{\Uc(G)}$ and 
$\Phi\circ \Psi=\id_{\Ec'_{\1}(G)} $. 

To this end let $D\in \Uc(G) $ arbitrary and denote $\Phi(D)=u $.
We have $u(f)=(D\check{f})(\1) $ and $\Psi(u)=D_u $, 
where 
$$(D_u f)(x)=(f\ast u)(x)=\check{u}(f\circ L_x)=D(f\circ L_x)(\1) 
=((Df)\circ L_x)(\1)=(Df)(x).$$
Hence $D_u f=Df $ and $D_u=D $ and we obtain   
$\Psi\circ  \Phi=\id_{\Uc(G)}$.

Now let $u\in \Ec'_{\1}(G) $ arbitrary. 
We have $\Psi(u)=D_u $.
Denote $\Phi(D_u)=v\in \Ec'_{\1}(G) $.

We have $v(f)=(D_u \check{f})(\1)=(\check{f} *u )(\1)=
\check{u}(\check{f}\circ L_{\1})=\check{u}(\check{f})=u(f) $.
Hence $v=u $ and $\Phi\circ \Psi=\id_{\Ec'_{\1}(G)} $ and 
the proof is complete.
\end{proof}

If $G$ is any pre-Lie group, then one can use Theorem~\ref{P 4.29} for endowing $\Uc(G)$ with a natural topology 
for which the map $\Psi$ is a homeomorphism if $\Ec'_{\1}(G)$ carries the weak dual topology 
which it gets as a closed linear subspace of $\Ec'(G)$ 
(see Definition~\ref{distrib}). 
That topology of $\Uc(G)$ can be equivalently described as the locally convex topology 
determined by the family of seminorms $\{D\mapsto \vert(Df)(\1)\vert\}_{f\in\Ec(G)}$. 

In the statement of the following corollary, we say that some Banach space $\Xc$ admits bump functions 
if there exists $\varphi\in\Ci(\Xc)$ which is equal to $1$ on some neighborhood of $0\in\Xc$, 
has the support contained in some ball, and for every $k\ge 1$ satisfies the condition 
$\sup\limits_{x\in\Xc}\Vert\de^k_x\phi\Vert<\infty$. 
Every Hilbert space admits bump functions; see \cite{WD73} and \cite{LW11} for more details and examples. 
In this setting, we will provide the following partial answer to Problem~\ref{probl1}. 

\begin{corollary}\label{density}
Let $G$ be any Banach-Lie group whose Lie algebra admits bump functions. 
Then $\Uc_0(G)$ is a dense subalgebra of $\Uc(G)$. 
\end{corollary}

\begin{proof}
By the Bahn-Banach theorem,  
it suffices to check that if $\theta\colon \Uc(G)\to\CC$ is any continuous linear functional 
 which vanishes on $\Uc_0(G)$, then $\theta=0$.  
To this end note that, by using the above family of seminorms describing 
the topology of $\Uc(G)$, 
one can find $f\in\Ec(G)$ with $\vert\theta(D)\vert\le\vert(Df)(\1)\vert$ for all $D\in\Uc(G)$. 
Then the kernel of the linear functional $D\mapsto (Df)(\1)$ is contained in $\Ker\theta$ 
and, since both these kernels are closed 1-codimensional subspaces of $\Uc(G)$,   
it follows that, after replacing $f$ by $cf$ for a suitable $c\in\CC$, 
we have  $\theta(D)=(Df)(\1)$ for all $D\in\Uc(G)$. 

The assumption $\theta(D)=0$ for all $D\in\Uc_0(G)$ is then equivalent 
to the fact that for all $k\ge 1$ we have $(\de^k(f\circ\exp_G))(0)=0$, 
where $\exp_G\colon\gg\to G$ is the exponential map of~$G$, which is a local diffeomorphism at $0\in\gg$. 
Now the hypothesis that the Lie algebra $\gg$ admits bump functions 
allows us to use \cite[Prop. 3]{LW11}, which ensures that 
for every local operator $T$ on $\gg$ we have $(T(f\circ\exp_G))(0)=0$. 

We will check that $(Df)(\1)=0$ for every local operator $D$ on~$G$. 
To this end pick some open sets $U$ and $V$ for which 
$\exp_G\colon V\to U$ is a diffeomorphism with the inverse denoted by $\log_G$, 
where $\1\in U\subseteq G$ and $0\in V\subseteq\gg$. 
Then use the hypothesis on~$\gg$ to find $\psi\in\Ci(\gg)$ with $\supp\psi\subseteq V$ and 
$\psi=1$ on some neighborhood of $0\in\gg$. 
Denote $\phi:=\psi\circ\log_G\in\Ci(U)$ and extend it with the value~$0$ on $G\setminus U$.  
Then $\phi\in\Ci(G)$, $\supp\phi\subseteq U$, 
 $\phi=1$ on some neighborhood of $\1\in U\subseteq G$, 
and $\psi=\phi\circ\exp_G$. 
Now define 
$$T\colon\Ci(\gg)\to\Ci(\gg), \quad Th=D(((\psi h)\circ\log_G)\phi)$$
where the function $(\psi h)\circ\log_G\in\Ci(V)$ is extended with the value $0$ on $\gg\setminus V$. 
Since $D$ is a local operator, it follows that also $T$ is a local operator, 
and then by the above observation we obtain $(T(f\circ\exp_G))(0)=0$, 
which is equivalent to $(Df)(\1)=0$. 
That is, $\theta(D)=0$ for arbitrary $D\in\Uc(G)$, 
and this concludes the proof. 
\end{proof}

\subsection*{Pre-Lie groups}
In order to illustrate the above general results and to relate them to the earlier literature, 
we conclude by some specific examples of topological groups which can be studied from the perspctive of the Lie theory 
(see also Remark~\ref{PBW_rem}). 
This is the case with the class of pre-Lie groups introduced in \cite{BR80} and \cite{BCR81},  
is closed with respect to several natural operations 
that may not preserve the locally compact or Lie groups, 
as for instance taking closed subgroups, infinite direct products, or projective limits (\cite[Prop. 1.3.1]{BCR81}). 
Some specific pre-Lie groups are briefly mentioned in Examples \ref{ex1}--\ref{ex4} below. 
See also \cite{HM07} and \cite{Gl02b} for further information on Lie theory for topological groups 
which may not be Lie groups. 

\begin{definition}
\label{Def3.1}
\normalfont 
A \emph{pre-Lie group} is any topological group $G$  
satisfying the conditions: 
\begin{enumerate}
\item The topological space $\Lg(G)$ has the structure of a locally convex Lie algebra over~$\RR$, 
whose scalar multiplication, vector addition and bracket 
satisfy the following conditions for all 
$t,s\in{\mathbb R}$ and $\gamma_1,\gamma_2\in\Lg(G)$: 
\begin{eqnarray}
\nonumber
 (t\cdot \gamma_1)(s) &=& \gamma_1(ts);  \\
\nonumber
(\gamma_1+\gamma_2)(t) &=& \lim\limits_{n\to\infty}(\gamma_1(t/n)\gamma_2(t/n))^n;\\
\nonumber
[\gamma_1,\gamma_2](t^2) &=& \lim\limits_{n\to\infty}(\gamma_1(t/n)\gamma_2(t/n)\gamma_1(-t/n)\gamma_2(-t/n))^{n^2}, 
\end{eqnarray}
where the convergence is assumed to be uniform on the compact subsets of $\RR$. 
\item 
For every nontrivial $\gamma\in\Lg(G)$ there exists a function $\varphi$ of class $\Ci$ on 
some neighborhood of $\1\in G$ such that $(D^\lambda_\gamma\varphi)(\1)\ne0$.  
\end{enumerate}
\end{definition}

\begin{remark}\label{obs 3.11}
 \normalfont
If $G$ is a pre-Lie group, then the multiplication mapping 
$m\colon G\times G\to G$, $(x,y)\mapsto xy$,  
is smooth by \cite[Th. 1.3.2.2 and subsect. 1.1.2]{BCR81} 
(or alternatively \cite[Th. and Sect. 1]{BR80}). 
In particular, by using the chain rule contained in 
condition~(dcm) of \cite[subsect. 1.3.2]{BCR81} 
(or alternatively the proof of (v) in \cite[Th.]{BR80}), 
we easily recover in this special case the result of Proposition~\ref{multi}, to the effect that 
that for any locally convex space $\Yc$ the linear mapping 
$\Ec(G,\Yc)\to\Ec(G\times G,\Yc)$, $\varphi\mapsto\varphi\circ m$, 
is well-defined and continuous. 
\end{remark}

\begin{example}\label{ex1}
\normalfont
Every locally compact group (in particular, every finite-dimensional Lie group) is a pre-Lie group 
(\cite[pag. 41--41]{BCR81}). 
In this special case our Theorem~\ref{P 4.29} agrees with \cite[Th. 1.4]{Ed88} and 
\cite[Cor. 2.6]{Ak95a}. 
\end{example}

\begin{example}\label{ex2}
\normalfont
Every Banach-Lie group is a pre-Lie group 
(\cite[pag. 41--41]{BCR81}). 
In this special case, we are unable to provide any earlier reference 
for the result of our Theorem~\ref{P 4.29}. 
\end{example}

\begin{example}\label{ex3}
\normalfont
If $\Xc$ is any locally convex space, then the abelian locally convex Lie group $(\Xc,+)$ 
is a pre-Lie group 
(\cite[pag. 41--41]{BCR81}). 
In this special case, we are again unable to provide any precise earlier reference 
for the result provided by our Theorem~\ref{P 4.29}. 
See however \cite[Th. 3.4]{Du73} for a related result on real Hilbert spaces. 
\end{example}

\begin{example}\label{ex4}
\normalfont
Other examples of locally convex Lie groups 
which are pre-Lie groups are the so-called groups of $\Gamma$-rapidly decreasing mappings 
with values in some Lie groups 
(\cite[Subsect. 4.2.2]{BCR81}). 

For the sake of completeness, we will briefly recall the construction of 
the aforementioned groups of rapidly decreasing mappings, in a very special situation. 
Let $n\ge 1$ be any integer and $\Gamma=\{\gamma_k\mid k\ge 0\}$, 
where 
$$(\forall k\ge 0)\quad \gamma_k(\cdot)=(1+\vert\cdot\vert)^k$$
and $\vert\cdot\vert$ stands for the Euclidean norm on $\RR^n$. 
Let $(\Ac,\Vert\cdot\Vert)$ be any unital associative Banach algebra with some fixed norm that defines its topology, 
and $\Ac^\times$ denote the group of invertible elements in $\Ac$, 
and consider the set of mappings 
$$\Sc(\RR^n,\Ac;\Gamma):=\{f\in\Ci(\RR^n,\Ac^\times)\mid 
\lim\limits_{\vert x\vert\to\infty}f(x)=\1;\ 
(\forall \alpha\in\NN^n)\quad \sup\limits_{\RR^n}\gamma_k(\cdot)\Vert\partial^\alpha f(\cdot)\Vert<\infty\} $$
endowed with the pointwise multiplication and inversion, 
where we denote by $\partial^\alpha$ the partial derivatives. 
Then the group of $\Gamma$-rapidly decreasing $\Ac^\times$-valued mappings  $\Sc(\RR^n,\Ac;\Gamma)$
 has the natural structure of a pre-Lie group. 
This follows as a very special case of 
 \cite[Cor. 4.1.1.7 and Th. 4.2.2.3]{BCR81}. 
\end{example}






\bigskip
\begin{thebibliography}{1000000}


\bibitem[Ak95]{Ak95a}
\textsc{S.S.~Akbarov}, 
Smooth structure and differential operators on a locally compact group. (Russian) 
\textit{Izv. Ross. Akad. Nauk Ser. Mat.} \textbf{59} (1995), no.~1, 3--48; 
translation in \textit{Izv. Math.} \textbf{59} (1995), no.~1, 1--44. 

\bibitem[Al13]{Al13}
\textsc{H.~Alzaareer}, 
\textit{Lie groups of mappings on
non-compact spaces and manifolds}. 
PhD thesis, Universit\"at Paderborn, 2013\\
(http://digital.ub.uni-paderborn.de/ubpb/urn/urn:nbn:de:hbz:466:2-11572). 

\bibitem[AD51]{AD51}
\textsc{R.~Arens, J.~Dugundji}, 
Topologies for function spaces.
\textit{Pacific J. Math.} \textbf{1} (1951), 5--31. 

\bibitem[Bel06]{Be06}
\textsc{D.~Belti\c t\u a}, 
\textit{Smooth Homogeneous Structures in Operator Theory}. 
Chapman \& Hall/CRC Monographs and Surveys in Pure and Applied Mathematics, 137. 
Chapman \& Hall/CRC, Boca Raton, FL, 2006.


\bibitem[BB11]{BB11}
\textsc{I.~Belti\c t\u a, D.~Belti\c t\u a},
On differentiability of vectors in Lie group representations.
\textit{J. Lie Theory} \textbf{21} (2011), no.~4, 771--785. 

\bibitem[Ber08]{Ber08}
\textsc{W.~Bertram}, 
\textit{Differential Geometry, Lie Groups and Symmetric Spaces over General Base Fields and Rings}. 
Mem. Amer. Math. Soc. 192 (2008), no. 900, x+202 pp.

\bibitem[BGN04]{BGN04}
\textsc{W.~Bertram, H.~Gl\"ockner, K.-H.~Neeb}, 
Differential calculus over general base fields and rings. 
\textit{Expo. Math.} \textbf{22} (2004), no. 3, 213--282.

\bibitem[BN05]{BN05}
\textsc{W.~Bertram, K.-H.~Neeb}, 
Projective completions of Jordan pairs. II. Manifold structures and symmetric spaces. 
\textit{Geom. Dedicata} \textbf{112} (2005), 73--113. 

\bibitem[BF66]{BF66}
\textsc{R.~Bonic, J.~Frampton}, 
Smooth functions on Banach manifolds. 
\textit{J. Math. Mech.} \textbf{15} (1966), 877--898.


\bibitem[BC73]{BC73} 
\textsc{H.~Boseck, G.~Czichowski}, 
Grundfunktionen und verallgemeinerte Funktionen auf topologischen Gruppen. I. 
\textit{Math. Nachr.} \textbf{58} (1973), 215--240. 

\bibitem[BC75]{BC75} 
\textsc{H.~Boseck, G.~Czichowski}, 
Grundfunktionen und verallgemeinerte Funktionen auf topologischen Gruppen. II. 
\textit{Math. Nachr.} \textbf{66} (1975), 319--332. 

\bibitem[BCR81]{BCR81} 
\textsc{H.~Boseck, G.~Czichowski, K.-P.~Rudolph}, 
\textit{Analysis on Topological Groups ---General Lie Theory}.  
Teubner-Texte zur Mathematik, 37. 
BSB B. G. Teubner Verlagsgesellschaft, Leipzig, 1981.

\bibitem[BR80]{BR80} 
\textsc{H.~Boseck, K.-P.~Rudolph}, 
An axiomatic approach to differentiability on topological groups.
\textit{Math. Nachr.} \textbf{98} (1980), 27--36.

\bibitem[Br61]{Br61}
\textsc{F.~Bruhat}, 
Distributions sur un groupe localement compact et applications 
\`a l'\'etude des repr\'esentations des groupes ${\mathfrak p}$-adiques. 
\textit{Bull. Soc. Math. France} \textbf{89} (1961), 43--75.

\bibitem[CW99]{CW99}
\textsc{A.~Cannas da Silva, A.~Weinstein}, 
\textit{Geometric Models for Noncommutative Algebras}. 
Berkeley Mathematics Lecture Notes, 10. 
American Mathematical Society, Providence, RI; 
Berkeley Center for Pure and Applied Mathematics, Berkeley, CA, 1999.

\bibitem[Du73]{Du73}
\textsc{D.N.~Dudin}, 
Generalized measures, or distributions, on Hilbert space. 
(Russian) 
\textit{Trudy Moskov. Mat. Ob\v s\v c.} \textbf{28} (1973), 135--158.

\bibitem[Ed88]{Ed88}
\textsc{T.~Edamatsu}, 
Differential operators on locally compact groups. 
\textit{J. Math. Kyoto Univ.} \textbf{28} (1988), no.~3, 445--503.

\bibitem[Eh56]{Eh56}
\textsc{L.~Ehrenpreis}, 
Some properties of distributions on Lie groups.  
\textit{Pacific J. Math.} \textbf{6} (1956), 591--605. 

\bibitem[Gl02a]{Gl02a}
\textsc{H.~Gl\"ockner}, 
Infinite-dimensional Lie groups without completeness restrictions. 
In: A.~Strasburger, J.~Hilgert, K.-H.~Neeb, W.~Wojty\'nski (eds.), 
``Geometry and analysis on finite- and infinite-dimensional Lie groups'' (B\c edlewo, 2000), 
Banach Center Publ., 55, Polish Acad. Sci., Warsaw, 2002, pp.~43--59.


\bibitem[Gl\"o02b]{Gl02b}
\textsc{H.~Gl\"ockner}, 
Real and $p$-adic Lie algebra functors on the category of topological groups. 
\textit{Pacific J. Math.} \textbf{203} (2002), no.~2, 321--368.

\bibitem[Gl\"o13]{Gl13}
\textsc{H.~Gl\"ockner}, 
Exponential laws for ultrametric partially differentiable functions and applications. 
{\it $p$-Adic Numbers Ultram. Anal. Appl.} {\bf 5} (2013), no.~2, 122--159.

\bibitem[Glu57]{Glu57}
\textsc{V.M.~Glu\v skov}, 
Lie algebras of locally bicompact groups. (Russian) 
\textit{Uspehi Mat. Nauk (N.S.)} \textbf{12} (1957), no.~2 (74), 137--142. 


\bibitem[HM05]{HM05}
\textsc{K.H.~Hofmann, S.A.~Morris}, 
Sophus Lie's third fundamental theorem and the adjoint functor theorem. 
\textit{J. Group Theory} \textbf{8} (2005), no.~1, 115--133.

\bibitem[HM07]{HM07}
\textsc{K.H.~Hofmann, S.A.~Morris}, 
\textit{The Lie Theory of Connected Pro-Lie Groups}. 
A Structure Theory for Pro-Lie Algebras, Pro-Lie Groups, 
and Connected Locally Compact Groups. 
EMS Tracts in Mathematics, 2. European Mathematical Society (EMS), 
Z\"urich, 2007.

\bibitem[HM13]{HM13}
\textsc{K.H.~Hofmann, S.A.~Morris}, 
\textit{The structure of compact groups}. 
A primer for the student ---a handbook for the expert.
Third edition, revised and augmented. De Gruyter Studies in Mathematics, 25. De Gruyter, Berlin, 2013. 

\bibitem[Ka59]{Ka59}
\textsc{G.I.~Kac}, 
Generalized functions on locally compact groups and the decomposition of a regular representation. 
(Russian)
\textit{Dokl. Akad. Nauk SSSR} \textbf{125} (1959), 27--30.

\bibitem[Ka61]{Ka61}
\textsc{G.I.~Kac}, 
Generalized functions on a locally compact group and decompositions of unitary representations. 
(Russian)
\textit{Trudy Moskov. Mat. Ob\v sc.} \textbf{10} (1961), 3--40. 

\bibitem[KM97]{KM97}
\textsc{A.~Kriegl, P.W.~Michor}, 
\textit{The Convenient Setting of Global Analysis}. 
Mathematical Surveys and Monographs, 53. American Mathematical Society, Providence, RI, 1997.

\bibitem[KMR14]{KMR14}
\textsc{A.~Kriegl, P.W.~Michor, A.~Rainer}, 
The convenient setting for Denjoy--Carleman differentiable mappings of Beurling and Roumieu type. 
\textit{Preprint} arXiv:1111.1819v2 [math.FA], 12 January 2014.

\bibitem[La57]{La57} 
\textsc{R.K.~Lashof}, 
Lie algebras of locally compact groups. 
\textit{Pacific J. Math.} \textbf{7} (1957), 1145--1162.

\bibitem[LW11]{LW11}
\textsc{D.H.~Leung, Y.-S.~Wang}, 
Local operators on $C^p$. 
\textit{J. Math. Anal. Appl.} \textbf{381} (2011), no.~1, 308--314.

\bibitem[Ma61]{Ma61}
\textsc{K.~Maurin}
Distributionen auf Yamabe-Gruppen. Harmonische Analyse auf einer Abelschen l. k. Gruppe. 
\textit{Bull. Acad. Polon. Sci. S\'er. Sci. Math. Astronom. Phys.} \textbf{9} (1961), 845--850.

\bibitem[MM64]{MM64}
\textsc{K.~Maurin, L.~Maurin}, 
Universelle umh\"ullende Algebra einer lokal kompakten Gruppe und ihre selbstadjungierte Darstellungen. Anwendungen. 
\textit{Studia Math.} \textbf{24} (1964), 227--243.

\bibitem[MM65]{MM65}
\textsc{K.~Maurin, L.~Maurin}, 
Automorphic forms spectral theory and group representations. 
\textit{Bull. Acad. Polon. Sci. S\'er. Sci. Math. Astronom. Phys.} \textbf{13} (1965), 199--203. 

\bibitem[Ne06]{Ne06}
\textsc{K.-H.~Neeb}, 
Towards a Lie theory of locally convex groups. 
\textit{Jpn. J. Math.} \textbf{1} (2006), no.~2, 291--468.

\bibitem[NS13a]{NS12} 
\textsc{K.-H.~Neeb, H.~Salmasian}, 
Differentiable vectors and unitary representations of Fr\'echet-Lie supergroups. 
\textit{Math. Z.} \textbf{275} (2013), no.~1-2, 419--451.

\bibitem[NS13b]{NS13} 
\textsc{K.-H.~Neeb, H.~Salmasian}, 
Positive definite superfunctions and unitary representations of Lie supergroups. 
\textit{Transform. Groups} \textbf{18} (2013), no. 3, 803--844.

\bibitem[Ped94]{Pe94}
\textsc{N.V.~Pedersen}, 
Matrix coefficients and a Weyl correspondence for nilpotent Lie groups. 
\textit{Invent. Math.} \textbf{118} (1994), no.~1, 1--36.

\bibitem[Pee60]{Pe60}
\textsc{J.~Peetre}, 
R\'ectification \`a l'article ``Une caract\'erisation abstraite des op\'erateurs diff\'erentiels''. 
\textit{Math. Scand.}  \textbf{8} (1960), 116--120.

\bibitem[Py74]{Py74}
\textsc{T.~Pytlik}, 
Nuclear spaces on a locally compact group. 
\textit{Studia Math.} \textbf{50} (1974), 225--243. 

\bibitem[Ri53]{Ri53}
\textsc{J.~Riss}, 
\'El\'ements de calcul diff\'erentiel et th\'eorie des distributions sur les groupes ab\'eliens localement compacts.  
\textit{Acta Math.} \textbf{89} (1953), 45--105. 

\bibitem[Wa72]{Wa72}
\textsc{G.~Warner}, 
\textit{Harmonic Analysis on Semi-Simple Lie Groups}. I. 
Die Grundlehren der mathematischen Wissenschaften, Band 188. 
Springer-Verlag, New York-Heidelberg, 1972.

\bibitem[WD73]{WD73}
\textsc{J.C.~Wells, C.R.~DePrima}, 
Local automorphisms are differential operators on some Banach spaces.  
\textit{Proc. Amer. Math. Soc.} \textbf{40} (1973), 453--457. 

\end{thebibliography}
\end{document}